\documentclass[letter,10pt]{amsart}
\usepackage{amsmath, amsfonts, amssymb,amscd,graphics,graphicx,color,eucal,setspace}
\usepackage{latexsym}
\usepackage{xcolor}
\usepackage{caption}
\usepackage{subcaption}
\usepackage{url}
\usepackage{float}

\usepackage{tikz-cd}
\usepackage{tikz}
\usetikzlibrary{snakes,
	3d, matrix,decorations.pathreplacing,calc,decorations.pathmorphing}
\usetikzlibrary{decorations.markings}
\usetikzlibrary {positioning}

\usepackage{caption}
\usepackage[labelformat=simple]{subcaption}

\numberwithin{equation}{section}

\theoremstyle{plain}
\newtheorem{theorem}{Theorem}[section]
\newtheorem{lemma}[theorem]{Lemma}
\newtheorem{corollary}[theorem]{Corollary}
\newtheorem{proposition}[theorem]{Proposition}

\theoremstyle{definition}
\newtheorem{remark}[theorem]{Remark}
\newtheorem{example}[theorem]{Example}
\newtheorem{definition}[theorem]{Definition}

\newtheorem{conjecture}[theorem]{Conjecture}

\def\ve{\mathbf{e}}

\def\C{\mathbb C}
\def\CP{\C P}
\def\R{\mathbb R}

\def\Z{\mathbb Z}

\def\GL{\mathrm{GL}}
\def\Fl{\mathcal{F}\ell}

\DeclareMathOperator{\rank}{rank}

\newcommand{\cfbox}[2]{%
    \colorlet{currentcolor}{.}%
    {\color{#1}%
    \fbox{\color{currentcolor}#2}}%
}

\begin{document}
\title[]{Toric Bruhat Interval Polytopes}

\date{\today}

\author[Eunjeong Lee]{Eunjeong Lee}
\address[E. Lee]{Center for Geometry and Physics, Institute for Basic Science (IBS), Pohang 37673, Republic of Korea}
\email{eunjeong.lee@ibs.re.kr}

\author[Mikiya Masuda]{Mikiya Masuda}
\address[M. Masuda]{Department of Mathematics, Graduate School of Science, Osaka City University, Sumiyoshi-ku, Sugimoto, 558-8585, Osaka, Japan}
\email{masuda@osaka-cu.ac.jp}

\author[Seonjeong Park]{Seonjeong Park}
\address[S. Park]{Department of mathematics, Ajou University,
Suwon 16499, Republic of Korea}
\email{seonjeong1124@gmail.com}



\maketitle

\begin{abstract}
For two elements $v$ and $w$ of the symmetric group $\mathfrak{S}_n$ with $v\leq w$ in Bruhat order, the Bruhat interval polytope $Q_{v,w}$ is the convex hull of the points $(z(1),\ldots,z(n))\in \R^n$ with $v\leq z\leq w$. It is known that the Bruhat interval polytope $Q_{v,w}$ is the moment map image of the Richardson variety $X^{v^{-1}}_{w^{-1}}$. We say that $Q_{v,w}$ is \emph{toric} if the corresponding Richardson variety $X_{w^{-1}}^{v^{-1}}$ is a toric variety.  We show that when $Q_{v,w}$ is toric, its combinatorial type is determined by the poset structure of the Bruhat interval $[v,w]$ while this is not true unless $Q_{v,w}$ is toric. We are concerned with the problem of when $Q_{v,w}$ is (combinatorially equivalent to) a cube because $Q_{v,w}$ is a cube if and only if  $X_{w^{-1}}^{v^{-1}}$ is a smooth toric variety.  We show that a Bruhat interval polytope $Q_{v,w}$ is a cube if and only if $Q_{v,w}$ is toric and the Bruhat interval $[v,w]$ is a Boolean algebra. We also give several sufficient conditions on $v$ and $w$ for $Q_{v,w}$ to be a cube.   
\end{abstract}

\setcounter{tocdepth}{1}
\tableofcontents

\section{Introduction}

The permutohedron $\mathrm{Perm}_{n-1}$ is an $(n-1)$-dimensional simple polytope in $\R^n$ defined by the convex hull of all points $(u(1),u(2),\ldots,u(n))\in\R^n$ for $u$ in the symmetric group $\mathfrak{S}_n$ on the set~$\{1,2,\dots,n\}$. It was first investigated by Schoute in 1911 (see~\cite{zieg98} and references therein), and later Guilbaud and Rosenstiehl gave the name ``permutohedron'' in~\cite{gu-ro63}.
There are many works on generalizations of the notion of permutohedra such as generalized permutohedra in~\cite{po09}, graphicahedra in~\cite{adlos10}, Bruhat interval polytopes in~\cite{ts-wi15}, and so on.

On the other hand, the permutohedra have  appeared in not only combinatorics but also the geometries of flag varieties.
The flag variety $\Fl_n$ is a smooth projective variety which consists of chains $\{0\}\subset V_1\subset V_2\subset \cdots \subset V_n=\C^n$ of  subspaces of $\C^n$ with $\dim_{\C} V_i=i$. It is known that the algebraic torus $\mathbb{T}=(\C^\ast)^n$ acts on~$\Fl_n$ and there is a one-to-one correspondence between the set of fixed points of $\mathbb{T}$ on $\Fl_n$ and the elements of $\mathfrak{S}_n$.
Moreover, the moment map image of $\Fl_n$ is the permutohedron~$\mathrm{Perm}_{n-1}$, and the closure of the $\mathbb{T}$-orbit of a generic point of $\Fl_n$ is known to be the permutohedral variety, which is the toric variety whose fan is the normal fan of $\mathrm{Perm}_{n-1}$ (see \cite{Klyachko_orbits_85, Procesi_toric_90}).

In this manuscript, we are studying Bruhat interval polytopes that were introduced by Tsukerman and Williams \cite{ts-wi15} in 2015. For two elements $v$ and $w$ of the symmetric group $\mathfrak{S}_n$ with $v\leq w$ in Bruhat order, the Bruhat interval polytope $Q_{v,w}$ is defined to be the convex hull of all points~$(z(1),\ldots,z(n))\in \R^n$ with $v\leq z\leq w$. Bruhat interval polytopes are one of the generalizations of permutohedra. Indeed, the Bruhat interval polytope $Q_{e, w_0}$ is the permutohedron $\textrm{Perm}_{n-1}$ where $e$ is the identity element and $w_0$ is the longest element in $\mathfrak{S}_n$.

As in the case of permutohedra and flag varieties, Bruhat interval polytopes are related with Richardson varieties.
For $v\le w$, the Richardson variety $X_w^v$ is defined to be the intersection of the Schubert variety $X_w$ and the opposite Schubert variety $w_0X_{w_0v}$.  It is an irreducible $\mathbb{T}$-invariant subvariety of the flag variety $\Fl_n$.  It is known that there is a one-to-one correspondence between the set of fixed points of $\mathbb{T}$ on the Richardson variety $X^v_w$ and the set $\{z\mid v\leq z\leq w\}$, and it leads naturally to consider the convex hull of the points $(z(1),\ldots,z(n))\in \R^n$ with $v\leq z\leq w$.  Note that the moment map image of the Richardson variety $X^v_w$ is the Bruhat interval polytope $Q_{v^{-1},w^{-1}}$ not $Q_{v,w}$ (see Lemma~\ref{lem:moment_map}).

It should be noted that Bruhat interval polytopes $Q_{v,w}$ and $Q_{v^{-1},w^{-1}}$ are not combinatorially equivalent in general even though the Bruhat intervals $[v,w]$ and $[v^{-1},w^{-1}]$ are isomorphic as posets. Moreover, even if two Bruhat interval polytopes $Q_{v,w}$ and $Q_{v^{-1},w^{-1}}$ are combinatorially equivalent, the fact that a subinterval $[x,y]\subset [v,w]$ is realized as a face of $Q_{v,w}$ does not imply that the subinterval $[x^{-1},y^{-1}]$ is realized as a face of $Q_{v^{-1},w^{-1}}$ (see Remark~\ref{rmk_BIP_inverses_are_not_combinatorially_equivalent}).

We are particularly interested in Bruhat interval polytopes whose corresponding Richardson varieties are toric varieties (note that $X_w^v$ is a toric variety if and only if so is $X_{w^{-1}}^{v^{-1}}$). We call such a Bruhat interval polytope \emph{toric}. 
It is known that dim $Q_{v,w} \leq \dim X^{v^{-1}}_{w^{-1}} = \ell(w) -\ell(v)$ in general, where $\ell(~)$ denotes the length of a permutation, and we have that $\dim Q_{v,w} = \ell(w) - \ell(v)$ if and only if $Q_{v,w}$ is toric (see Section 3). 
Toric Bruhat interval polytopes have a bunch of nice properties that an arbitrary Bruhat interval polytope does not have. Furthermore, those nice properties give us topological and geometric information of toric Richardson varieties.

\begin{theorem}[{Theorem~\ref{prop:3-2}}]
	A Bruhat interval polytope $Q_{v,w}$ is toric if and only if every subinterval $[x,y]$ of $[v,w]$ is realized as a face of $Q_{v,w}$.
\end{theorem}

The above theorem implies that if $Q_{v,w}$ is toric, then its combinatorial type is  determined by the poset structure of $[v,w]$, and hence $Q_{v,w}$ and $Q_{v^{-1},w^{-1}}$ are combinatorially equivalent.

Combinatorial properties of a toric Bruhat interval polytope $Q_{v,w}$ give us some geometric information about the toric Richardson variety $X^v_w$. The toric Richardson variety $X^v_w$ is smooth at a $\mathbb{T}$-fixed point $uB$ for $v\leq u\leq w$ if and only if the vertex $u$ of the Bruhat interval polytope $Q_{v,w}$ is simple, that is, the number of edges meeting at the vertex $u$ is same as the dimension of the polytope $Q_{v,w}$ (see Proposition~\ref{prop:smooth}). Hence a Richardson variety is a smooth toric variety if and only if the corresponding Bruhat interval polytope is toric and a simple polytope.

Note that every toric Schubert variety is smooth and its corresponding Bruhat interval polytope is combinatorially equivalent to a cube (see~\cite{Fan98,Karu_Schubert_13,le-ma18}). But not every toric Bruhat interval polytope is a simple polytope and hence not every toric Richardson variety is smooth. See Figures~\ref{fig:S4 and 4-crown} and~\ref{fig:1324-3412}. By restricting our attention to toric Bruhat interval polytopes, we get the following.

\begin{proposition}[{Proposition~\ref{prop:3-5}}]
A toric Bruhat interval polytope is a simple polytope if and only if it is combinatorially equivalent to a cube.
\end{proposition}

It is well-known in toric topology that every smooth toric variety whose fan is the normal fan of a combinatorial cube has a sequence of $\CP^1$-fiber bundles, so called a {Bott tower}.\footnote{A \emph{Bott tower} is a family of smooth projective toric varieties $\{B_{2k} \mid 1\leq k \leq n \}$ such that $B_2=\CP^1$ and $B_{2k}=P(\underline{\C}\oplus \xi_{k-1})$ for $1< k\leq n$ where $P(\cdot)$ denotes complex projectivization, $\xi_{k-1}$ is a complex line bundle over $B_{2(k-1)}$ and $\underline{\C}$ is the trivial line bundle (see~\cite{gr-ka94}). We call $B_{2k}$ a \emph{Bott manifold} (of height~$k$).} Hence the above proposition implies that every smooth toric Richardson variety is a Bott manifold that is a manifold in a Bott tower.
We can further show the following whose geometric meaning is that a Richardson variety $X^v_w$ is a Bott manifold if and only if it is toric and the Bruhat interval $[v,w]$ is a Boolean algebra.

\begin{theorem}[{Theorem~\ref{theo:3-6}}]\label{thm:main}
A Bruhat interval polytope $Q_{v,w}$ is combinatorially equivalent to a cube if and only if it is toric and the Bruhat interval $[v,w]$ is a Boolean algebra.
\end{theorem}

In the above theorem, we cannot drop the toric condition. There exist permutations $v$ and $w$ in~$\mathfrak{S}_n$ ($n\geq 4$) such that the Bruhat interval $[v,w]$ is a Boolean algebra but the combinatorial type of the Bruhat interval polytope $Q_{v,w}$ is not a cube. See Figure~\ref{fig:1324-4231} and Section~\ref{sec:Product of Bruhat intervals}.

We also study a necessary and sufficient condition on $v$ and $w$ such that the Bruhat interval polytope~$Q_{v,w}$ is toric or combinatorially equivalent to a cube.
It was shown in~\cite{Fan98} that a Bruhat interval polytope $Q_{e,w}$ is combinatorially equivalent to a cube if and only if $w$ is a product of distinct simple transpositions. But the similar extension does not hold for general $v$. That is, even if there exist reduced expressions $r(v)$ and $r(w)$ for $v$ and $w$ such that the subword $r(w)\setminus r(v)$ of $r(w)$ consists of distinct simple transpositions, we cannot conclude that $Q_{v,w}$ is combinatorially equivalent to a cube (see Example~\ref{exam:3-1}) nor toric (see Example~\ref{ex:not-distinct-cube}). So, it seems difficult to characterize $v$ and $w$ for which $Q_{v,w}$ is toric or combinatorially equivalent to a cube. We find some sufficient conditions on $v$ and $w$ for $Q_{v,w}$ to be toric, and give a necessary and sufficient condition on $v$ and $w$ for $Q_{v,w}$ to be a cube when $v$ and $w$ satisfy some special condition.

This manuscript is organized as follows. In Section~\ref{sec:Preliminaries}, we compile some basic facts on posets, polytopes and toric varieties, and introduce Bruhat interval polytopes. In Section~\ref{sec:Relationship with Richardson varieties}, we show that the Bruhat interval polytope $Q_{v,w}$ is the moment map image of the Richardson variety $X^{v^{-1}}_{w^{-1}}.$ In Section~\ref{sec:Properties of Bruhat interval polytopes}, we interpret combinatorial properties of Bruhat interval polytopes in terms of graphs defined by  Bruhat intervals. Section~\ref{sec:Toric Bruhat interval polytopes} deals with properties of toric Bruhat interval polytopes and contains the proof of Theorem~\ref{thm:main}. In Section~\ref{sec:Product of Bruhat intervals}, we show that there are infinitely many non-simple toric Bruhat interval polytopes. In Section~\ref{sec:Conditions on v and w}, we find some sufficient conditions on $v$ and $w$ for $Q_{v,w}$ to be toric, and then for such toric Bruhat interval polytopes $Q_{v,w}$ we find a sufficient condition to be a cube. In Section~\ref{sec:Finding all coatoms},  we will find all coatoms of the Bruhat interval $[v,w]$ when $v$ and $w$ satisfy some special condition, and then describe when $Q_{v,w}$ is a cube for such special cases.

\bigskip

\noindent\textbf{Acknowledgements.}
The authors thanks to Akiyoshi Tsuchiya for his computer program to check Conjecture~\ref{conj} for $\mathfrak{S}_5$ and $\mathfrak{S}_6$.
Lee was supported by IBS-R003-D1. Masuda was supported in part by JSPS Grant-in-Aid for Scientific Research 16K05152.
Park was supported by Basic Science Research Program through the National Research Foundation of Korea (NRF) funded by the Government of Korea (NRF-2018R1A6A3A11047606) and (NRF-2016R1D1A1A09917654).

\section{Preliminaries}\label{sec:Preliminaries}
In this section, we prepare some notions and basic facts about posets and polytopes, and then introduce the notion of Bruhat interval polytopes.

\subsection{Posets and Bruhat orders} Let $\mathcal{P}$ be a poset (partially ordered set) with an order relation~$<$.
For two elements $x,y\in\mathcal{P}$, we say $y$ \emph{covers} $x$, denoted by $x\lessdot y$, if $x<y$ and there is no $z$ such that $x<z<y$. We also call it a cover $x\lessdot y$.
One represents $\mathcal{P}$ as a  mathematical diagram, called a \emph{Hasse diagram}, in a way that a point in the plane is drawn for each element of~$\mathcal{P}$, and a line segment or curve is drawn  upward from $x$ to $y$ whenever $y$ covers $x$.
A \emph{chain} of~$\mathcal{P}$ is
a totally ordered subset
 $\sigma$ of $\mathcal{P}$, and the \emph{length} $\ell(\sigma)$ of a chain~$\sigma$ is defined to be $|\sigma|-1$.
The \emph{length} $\ell(\mathcal{P})$ of a poset $\mathcal{P}$ is the length of a longest chain of~$\mathcal{P}$.
For $x \leq y$ in~$\mathcal{P}$, let $[x,y]$ denote the closed interval $\{z\in \mathcal{P}\mid x \le z \le y\}$, and let $(x,y)$ denote the open interval $\{z\in\mathcal{P}\mid x<z<y\}$.
If $\mathcal{P}$ has a unique minimum element, it is referred
to as the bottom element. Similarly, the unique maximum element, if it exists,
is referred to as the top element.
An element of~$\mathcal{P}$ that covers the bottom element is called  an \emph{atom}; and an element covered by the top element is called a \emph{coatom}.

A \emph{graded poset} is a poset $\mathcal{P}$ equipped with a rank function $\rho$ from $\mathcal{P}$ to $\Z_{\geq 0}$ satisfying the following:
\begin{enumerate}
\item if $x< y$ in $\mathcal{P}$, then $\rho(x)< \rho(y)$; and
\item if $x\lessdot y$, then $\rho(y)=\rho(x)+1$.
\end{enumerate}
The value of the rank function for an element of the poset is called its \emph{rank}.

 Let $\mathfrak{S}_n$ be the symmetric group on the set $[n]:= \{1,2,\dots,n\}$. We will denote an element $v\in \mathfrak{S}_n$ by $$[v(1),v(2),\ldots,v(n)]\quad\text{ or }\quad v(1)v(2)\cdots v(n).$$
For $1 \leq i < j \leq n$, the permutation which acts on $[n]$ by swapping $i$ and $j$ is called a \emph{transposition} and denoted by $(i,j)$ or $t_{i,j}$. Indeed,
 \[
 \begin{tikzcd}[row sep = 0cm, column sep = -0.2cm]
 (i,j) = t_{i,j} = [1,2,\dots,i-1, &j,& i+1,\dots,j-1,&i,&j+1,\dots,n]. \\
 &i\text{th}&&j\text{th}
 \end{tikzcd}
 \]
 We denote the set of transpositions in $\mathfrak{S}_n$ by $T$.
\begin{equation}\label{eq_def_of_T}
T=\{(i,j)\mid 1\leq i < j\leq n\}.
\end{equation}
The \emph{simple transpositions} $s_i$ are the transpositions of the form
\[
s_i := (i,i+1), \quad \text{ for }i=1,\ldots,n-1.
\]
Note that every element of $\mathfrak{S}_n$ can be represented as a
product of simple transpositions, although the decomposition is not
unique.

For $v\in\mathfrak{S}_n$, if $v= s_{i_1}\cdots s_{i_\ell}$ and  is minimal among all such expressions, then the string of indices $i_1 \cdots i_\ell$ is called a \emph{reduced decomposition} of~$v$ and $\ell$ is called the \emph{length} of~$v$, denoted $\ell(v)$. Note that $\ell(v^{-1})=\ell(v)$.
The Bruhat order on $\mathfrak{S}_n$ is defined by $v\leq w$  if a reduced decomposition for $v$ is a substring of some reduced decomposition for $w$. Then the Bruhat order on $\mathfrak{S}_n$ is a graded poset, with rank function given by length. The elements
\[
e:=[1,2,\ldots,n] \quad \text{ and } \quad w_0:=[n,n-1,\ldots,1]
\]
are the bottom and the top elements of the poset $\mathfrak{S}_n$, respectively.
For $v$ and $w$ in $\mathfrak{S}_n$ with $v\leq w$, the \emph{Bruhat interval} $[v,w]$ is defined to be the closed interval $$[v,w]=\{z\in\mathfrak{S}_n\mid v\leq z\leq w\}.$$
Figure~\ref{fig:S_4} shows the Hasse diagram of $\mathfrak{S}_4$ under Bruhat order and we illustrate an example of a Bruhat interval in Figure~\ref{fig:interval_1324-3412}
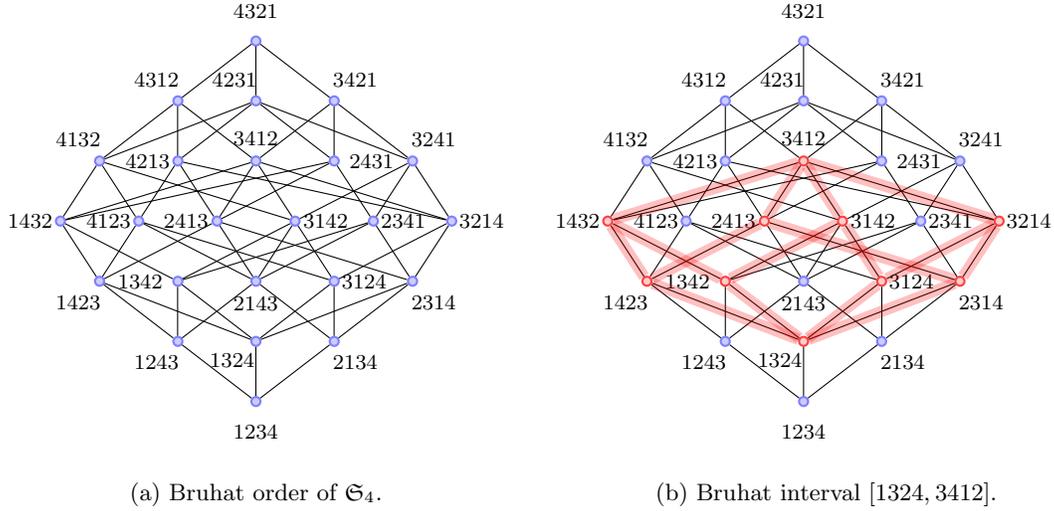
\begin{figure}[h!]
\begin{subfigure}[b]{0.49\textwidth}
\centering
\begin{tikzpicture}[scale=.7]
\tikzset{every node/.style={font=\footnotesize}}
		\matrix [matrix of math nodes,column sep={0.52cm,between origins},
		row sep={0.8cm,between origins},
		nodes={circle, draw=blue!50,fill=blue!20, thick, inner sep = 0pt , minimum size=1.2mm}]
		{
			& & & & & \node[label = {above:{4321}}] (4321) {} ; & & & & & \\
			& & &
			\node[label = {above left:4312}] (4312) {} ; & &
			\node[label = {above left:4231}] (4231) {} ; & &
			\node[label = {above right:3421}] (3421) {} ; & & & \\
			& \node[label = {above left:4132}] (4132) {} ; & &
			\node[label = {left:4213}] (4213) {} ; & &
			\node[label = {[label distance = -0.1cm] above:3412}] (3412) {} ; & &
			\node[label = {[label distance = 0.1cm]0:2431}] (2431) {} ; & &
			\node[label = {above right:3241}] (3241) {} ; & \\
			\node[label = {left:1432}] (1432) {} ; & &
			\node[label = {left:4123}] (4123) {} ; & &
			\node[label = {[label distance = 0.01cm]180:2413}] (2413) {} ; & &
			\node[label = {[label distance = 0.01cm]0:3142}] (3142) {} ; & &
			\node[label = {right:2341}] (2341) {} ; & &
			\node[label = {right:3214}] (3214) {} ; \\
			& \node[label = {below left:1423}] (1423) {} ; & &
			\node[label = {[label distance = 0.1cm]182:1342}] (1342) {} ; & &
			\node[label = {[label distance = -0.1cm] below:2143}] (2143) {} ; & &
			\node[label = {right:3124}] (3124) {} ; & &
			\node[label = {below right:2314}] (2314) {} ; & \\
			& & & \node[label = {below left:1243}] (1243) {} ; & &
			\node[label = {[label distance = 0.01cm]190:1324}] (1324) {} ; & &
			\node[label = {below right:2134}] (2134) {} ; & & & \\
			& & & & & \node[label = {below:1234}] (1234) {} ; & & & & & \\
		};		
		\draw (4321)--(4312)--(4132)--(1432)--(1423)--(1243)--(1234)--(2134)--(2314)--(2341)--(3241)--(3421)--(4321);
		\draw (4321)--(4231)--(4132);
		\draw (4231)--(3241);
		\draw (4231)--(2431);
		\draw (4231)--(4213);
		\draw (4312)--(4213)--(2413)--(2143)--(3142)--(3241);
		\draw (4312)--(3412)--(2413)--(1423)--(1324)--(1234);
		\draw (3421)--(3412)--(3214)--(3124)--(1324);
		\draw (3421)--(2431)--(2341)--(2143)--(2134);
		\draw (4132)--(4123)--(1423);
		\draw (4132)--(3142)--(3124)--(2134);
		\draw (4213)--(4123)--(2143)--(1243);
		\draw (4213)--(3214);
		\draw (3412)--(1432)--(1342)--(1243);
		\draw (2431)--(1432);
		\draw (2431)--(2413)--(2314);
		\draw (3142)--(1342)--(1324);
		\draw (4123)--(3124);
		\draw (2341)--(1342);
		\draw (2314)--(1324);
		\draw (3412)--(3142);
		\draw (3241)--(3214)--(2314);
\end{tikzpicture}
\subcaption{Bruhat order of $\mathfrak{S}_4$.}\label{fig:S_4}
\end{subfigure}
\begin{subfigure}[b]{0.49\textwidth}
		\begin{tikzpicture}[scale=.7]
			\tikzset{every node/.style={font=\footnotesize}}
			\tikzset{red node/.style = {fill=red!20!white, draw=red!75!white}}
	\tikzset{red line/.style = {line width=1ex, red,nearly transparent}}	
		\matrix [matrix of math nodes,column sep={0.52cm,between origins},
		row sep={0.8cm,between origins},
		nodes={circle, draw=blue!50,fill=blue!20, thick, inner sep = 0pt , minimum size=1.2mm}]
		{
			& & & & & \node[label = {above:{4321}}] (4321) {} ; & & & & & \\
			& & &
			\node[label = {above left:4312}] (4312) {} ; & &
			\node[label = {above left:4231}] (4231) {} ; & &
			\node[label = {above right:3421}] (3421) {} ; & & & \\
			& \node[label = {above left:4132}] (4132) {} ; & &
			\node[label = {left:4213}] (4213) {} ; & &
			\node[label = {[label distance = -0.1cm] above:3412}, red node] (3412) {} ; & &
			\node[label = {[label distance = 0.1cm]0:2431}] (2431) {} ; & &
			\node[label = {above right:3241}] (3241) {} ; & \\
			\node[label = {left:1432}, red node] (1432) {} ; & &
			\node[label = {left:4123}] (4123) {} ; & &
			\node[label = {[label distance = 0.01cm]180:2413}, red node] (2413) {} ; & &
			\node[label = {[label distance = 0.01cm]0:3142}, red node] (3142) {} ; & &
			\node[label = {right:2341}] (2341) {} ; & &
			\node[label = {right:3214}, red node] (3214) {} ; \\
			& \node[label = {below left:1423}, red node] (1423) {} ; & &
			\node[label = {[label distance = 0.1cm]182:1342}, red node] (1342) {} ; & &
			\node[label = {[label distance = -0.1cm] below:2143}] (2143) {} ; & &
			\node[label = {right:3124}, red node] (3124) {} ; & &
			\node[label = {below right:2314}, red node] (2314) {} ; & \\
			& & & \node[label = {below left:1243}] (1243) {} ; & &
			\node[label = {[label distance = 0.01cm]190:1324}, red node] (1324) {} ; & &
			\node[label = {below right:2134}] (2134) {} ; & & & \\
			& & & & & \node[label = {below:1234}] (1234) {} ; & & & & & \\
		};
		
		\draw (4321)--(4312)--(4132)--(1432)--(1423)--(1243)--(1234)--(2134)--(2314)--(2341)--(3241)--(3421)--(4321);
		\draw (4321)--(4231)--(4132);
		\draw (4231)--(3241);
		\draw (4231)--(2431);
		\draw (4231)--(4213);
		\draw (4312)--(4213)--(2413)--(2143)--(3142)--(3241);
		\draw (4312)--(3412)--(2413)--(1423)--(1324)--(1234);
		\draw (3421)--(3412)--(3214)--(3124)--(1324);
		\draw (3421)--(2431)--(2341)--(2143)--(2134);
		\draw (4132)--(4123)--(1423);
		\draw (4132)--(3142)--(3124)--(2134);
		\draw (4213)--(4123)--(2143)--(1243);
		\draw (4213)--(3214);
		\draw (3412)--(1432)--(1342)--(1243);
		\draw (2431)--(1432);
		\draw (2431)--(2413)--(2314);
		\draw (3142)--(1342)--(1324);
		\draw (4123)--(3124);
		\draw (2341)--(1342);
		\draw (2314)--(1324);
		\draw (3412)--(3142);
		\draw (3241)--(3214)--(2314);
		\draw[red line] (1324)--(1423)--(1432)--(3412);
		\draw[red line] (1324)--(2314)--(3214)--(3412);
		\draw[red line] (1324)--(1342)--(1432);
		\draw[red line] (2314)--(2413)--(3412);
		\draw[red line] (1423)--(2413);
		\draw[red line] (1342)--(3142)--(3412);
		\draw[red line] (1324)--(3124)--(3214);
		\draw[red line] (3124)--(3142);
		\end{tikzpicture}
		\subcaption{Bruhat interval $[1324, 3412]$.}\label{fig:interval_1324-3412}
\end{subfigure}
\caption{Bruhat order of $\mathfrak{S}_4$ and an example of a Bruhat interval.}\label{fig:S4 and 4-crown}
\end{figure}

\subsection{Polytopes and toric varieties} A \emph{convex polytope} is the convex hull  of a finite set of points in the Euclidean space $\R^n$. It is well known that every convex polytope is a bounded intersection of finitely many half-spaces. Two polytopes are \emph{combinatorially equivalent} if their face posets are isomorphic.
For a vertex~$v$ of a polytope $P$, the \emph{degree} $d(v)$ of $v$ is the number of edges meeting at~$v$.
For an $n$-dimensional polytope $P$, a vertex~$v$ of~$P$ is said to be \emph{simple} if $d(v) = n$. When all the vertices of~$P$ are simple, we call $P$ a \emph{simple polytope}.

A \emph{lattice polytope} is a convex polytope whose vertices are in the lattice $\Z^n\subset \R^n$. A vertex $v$ of a lattice polytope $P$ is said to be \emph{smooth} if it is simple and the primitive direction vectors of the edges emanating from $v$ form a basis for $\Z^n$. We call a vertex of~$P$ \emph{singular} if it is not smooth.  A lattice polytope $P$ is said to be \emph{smooth} if all the vertices of~$P$ are smooth. We call a lattice polytope $P$ is \emph{singular} if some vertex of $P$ is singular.
See Figure~\ref{fig:def-poly}.

\begin{figure}[h]
    \begin{subfigure}[b]{.3\textwidth}
        \begin{center}
            \begin{tikzpicture}[scale=0.4]
                \filldraw[draw=black,fill=lightgray] (0,0)--(2,-1)--(3,0.5)--(1.5,2)--cycle;
                \draw[dotted](0,0)--(1.2,0.8)--(3,0.5);
                \draw[dotted] (1.5,2)--(1.2,0.8);
                \draw (2,-1)--(1.5,2);
            \end{tikzpicture}
        \end{center}
    \caption{Not simple (so singular).}
        \end{subfigure}
        \begin{subfigure}[b]{.3\textwidth}
        \begin{center}
            \begin{tikzpicture}[scale=0.4]
                \filldraw[draw=black,fill=lightgray] (0,0)--(4,0)--(0,2)--cycle;
                \foreach \x in {-1,0,...,5}
                    \foreach \y in {-1,0,...,3}
                    {
                    \fill (\x,\y) circle (2pt);
                    }
            \end{tikzpicture}
        \end{center}
    \caption{Simple but singular.}
        \end{subfigure}
        \begin{subfigure}[b]{.3\textwidth}
        \begin{center}
            \begin{tikzpicture}[scale=0.4]
                \filldraw[draw=black,fill=yellow] (0,0)--(2,0)--(0,2)--cycle;
                \foreach \x in {-1,0,...,3}
                    \foreach \y in {-1,0,...,3}
                    {
                    \fill (\x,\y) circle (2pt);
                    }
            \end{tikzpicture}
        \end{center}
    \caption{Smooth.}
    \end{subfigure}
    \caption{Example and non-examples of smooth lattice polytopes.}
    \label{fig:def-poly}
    \end{figure}
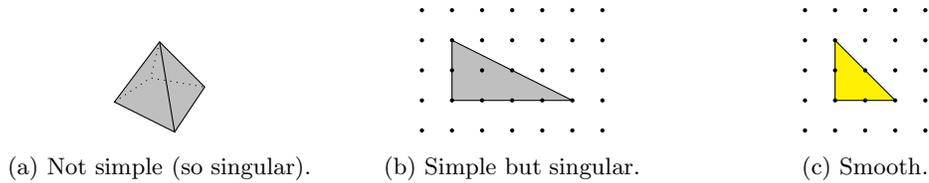

A \emph{toric variety} of complex dimension $n$ is a normal algebraic variety containing an algebraic torus $(\C^\ast)^n$ as a Zariski open dense subset such that the action of the torus on itself extends to the whole variety. It is known that a lattice polytope $P$ defines a projective toric variety $X(P)$, that is, $X(P)$ can be given as the closure of the image of a map $(\C^\ast)^n\to \CP^\ell$,
defined by Laurent monomials as in~\cite[Proposition 3.1.6]{CLStoric11}. Moreover, the vertices of~$P$ correspond to the $\mathbb{T}$-fixed points of $X(P)$, and a vertex~$v$ of~$P$ is smooth if and only if $X(P)$ is smooth at the corresponding fixed point.

It was shown in~\cite[Corollary 3.5]{ma-pa08} that if a smooth lattice polytope $P$ is combinatorially equivalent to a cube, then the toric variety $X(P)$ is weakly equivariantly diffeomorphic to a Bott manifold (see the footnote in the introduction for Bott manifolds).

\subsection{Bruhat interval polytope} The notion of Bruhat interval polytopes was introduced by Tsukerman and Williams~\cite{ts-wi15} as a natural generalization of permutohedra.
\begin{definition}
For elements $v$ and $w$ in $\mathfrak{S}_{n}$ with $v\leq w$, the \emph{Bruhat interval polytope} $Q_{v,w}$ is the convex hull of all permutation vectors $z=(z(1),z(2),\ldots,z(n))\in\R^n$ with $v\leq z\leq w$.
\end{definition}
 By definition, every Bruhat interval polytope is a lattice polytope and hence it defines a projective toric variety.

Note that the Bruhat interval polytope $Q_{e,w_0}$ is the permutohedron $\mathrm{Perm}_{n-1}$, the convex hull of the $n!$ points obtained by permuting the coordinates of the vector $(1,2,\dots,n)$. Two vertices $(v(1),\ldots,v(n))$ and $(w(1),\ldots,w(n))$ are joined by an edge in the permutohedron if and only if there exists a simple transposition $s_i$ such that $w=s_iv$ (see Figure~\ref{fig_Perm3}). Furthermore, the permutohedron~$\mathrm{Perm}_{n-1}$ defines a  smooth projective toric variety called the \emph{permutohedral variety}. But not every toric variety defined by a Bruhat interval polytope is smooth. For example, the polytope $Q_{1324,3412}$ is not a simple polytope (see Figure~\ref{fig_BIP_1324-3412}).
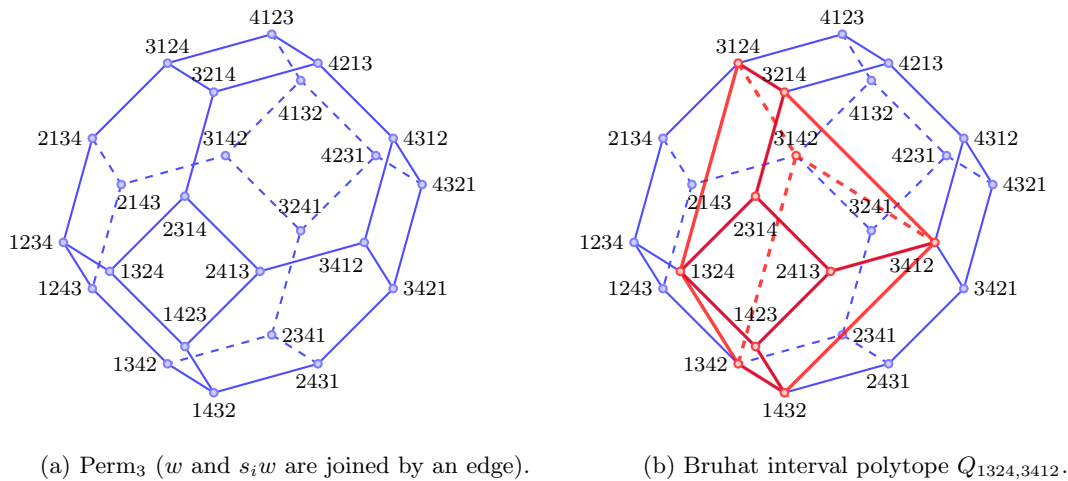
\begin{figure}[H]
\begin{subfigure}[b]{0.49\textwidth}
	\begin{tikzpicture}[scale=6]
	\tikzset{every node/.style={draw=blue!50,fill=blue!20, circle, thick, inner sep=1pt,font=\footnotesize}}
	
\coordinate (3142) at (1/3, 1/2, 1/6);
\coordinate (4231) at (2/3, 1/2, 1/6);
\coordinate (4312) at (5/6, 2/3, 1/2);
\coordinate (4321) at (5/6, 1/2, 1/3);
\coordinate (3421) at (5/6, 1/3, 1/2);
\coordinate (4213) at (2/3, 5/6, 1/2);
\coordinate (1324) at (1/3, 1/2, 5/6);
\coordinate (2413) at (2/3, 1/2, 5/6);
\coordinate (3412) at (5/6, 1/2, 2/3);
\coordinate (2314) at (1/2, 2/3, 5/6);
\coordinate (4123) at (1/2, 5/6, 1/3);
\coordinate (4132) at (1/2, 2/3, 1/6);
\coordinate (3214) at (1/2, 5/6, 2/3);
\coordinate (3124) at (1/3, 5/6, 1/2);
\coordinate (2431) at (2/3, 1/6, 1/2);
\coordinate (1432) at (1/2, 1/6, 2/3);
\coordinate (1423)  at (1/2, 1/3, 5/6);
\coordinate (1342)  at (1/3, 1/6, 1/2);
\coordinate (2341) at (1/2, 1/6, 1/3);
\coordinate (3241) at (1/2, 1/3, 1/6);
\coordinate (1243) at (1/6, 1/3, 1/2);
\coordinate (2143) at (1/6, 1/2, 1/3);
\coordinate (1234) at (1/6, 1/2, 2/3);
\coordinate (2134) at (1/6, 2/3, 1/2);

\draw[thick, draw=blue!70] (4213)--(4312)--(3412)--(2413)--(2314)--(3214)--cycle;
\draw[thick, draw=blue!70] (4312)--(4321)--(3421)--(3412);
\draw[thick, draw=blue!70] (3421)--(2431)--(1432)--(1423)--(2413);
\draw[thick, draw=blue!70] (1423)--(1324)--(2314);
\draw[thick, draw=blue!70] (1432)--(1342)--(1243)--(1234)--(1324);
\draw[thick, draw=blue!70] (1234)--(2134)--(3124)--(3214);
\draw[thick, draw=blue!70] (3124)--(4123)--(4213);

\draw[thick, draw=blue!70, dashed] (2134)--(2143)--(3142)--(4132)--(4123);
\draw[thick, draw=blue!70, dashed] (2143)--(1243);
\draw[thick, draw=blue!70, dashed] (3142)--(3241)--(2341)--(1342);
\draw[thick, draw=blue!70, dashed] (2341)--(2431);
\draw[thick, draw=blue!70, dashed] (3241)--(4231)--(4132);
\draw[thick, draw=blue!70, dashed] (4231)--(4321);

\node [label = {[label distance = 0cm]left:1234}] at (1234) {};
\node[label = {[label distance = 0cm]left:1243}] at (1243) {};
\node[label = {[label distance = 0cm]right:1324}] at (1324) {};
\node[label = {[label distance = 0cm]left:1342}] at (1342) {};
\node [label = {[label distance = 0cm]above:1423}] at (1423) {};
\node[label = {[label distance = -0.2cm]below:1432}] at (1432) {};
\node [label = {[label distance = 0cm]left:2134}] at (2134) {};
\node[label = {[label distance = -0.1cm]below right:2143}] at (2143) {};
\node[label = {[label distance = 0cm]below:2314}] at (2314) {};
\node[label = {[label distance = 0cm]right:2341}] at (2341) {};
\node[label = {[label distance = 0cm]left:2413}] at (2413) {};
\node[label = {[label distance = -0.2cm]below:2431}] at (2431) {};
\node[label = {[label distance = -0.2cm]above:3124}] at (3124) {};
﻿\node[label = {[label distance = -0.2cm]above:3142}] at (3142) {};
\node[label = {[label distance = -0.2cm]above:3214}] at (3214) {};
\node [label = {[label distance = -0.1cm]above:3241}] at (3241) {};
\node[label = {[label distance = 0cm]below left:3412}] at (3412) {};
\node[label = {[label distance = 0cm]right:3421}] at (3421) {};
\node[label = {[label distance = -0.2cm]above:4123}] at (4123) {};
\node [label = {[label distance = 0cm]below:4132}] at (4132) {};
\node[label = {[label distance = 0cm]right:4213}] at (4213) {};
\node[label = {[label distance = 0cm]left:4231}] at (4231) {};
\node[label = {[label distance = 0cm]right:4312}] at (4312) {};
\node [label = {[label distance = 0cm]right:4321}] at (4321) {};
	\end{tikzpicture}
	\subcaption{$\mathrm{Perm}_3$ ($w$ and $s_iw$ are joined by an edge).}
	\label{fig_Perm3}
\end{subfigure}
\begin{subfigure}[b]{0.49\textwidth}
	\begin{tikzpicture}[scale=6]
		\tikzset{every node/.style={draw=blue!50,fill=blue!20, circle, thick, inner sep=1pt,font=\footnotesize}}
\tikzset{red node/.style = {fill=red!20!white, draw=red!75!white}}
\tikzset{red line/.style = {line width=0.3ex, red, nearly opaque}}
	
	\coordinate (3142) at (1/3, 1/2, 1/6);
	\coordinate (4231) at (2/3, 1/2, 1/6);
	\coordinate (4312) at (5/6, 2/3, 1/2);
	\coordinate (4321) at (5/6, 1/2, 1/3);
	\coordinate (3421) at (5/6, 1/3, 1/2);
	\coordinate (4213) at (2/3, 5/6, 1/2);
	\coordinate (1324) at (1/3, 1/2, 5/6);
	\coordinate (2413) at (2/3, 1/2, 5/6);
	\coordinate (3412) at (5/6, 1/2, 2/3);
	\coordinate (2314) at (1/2, 2/3, 5/6);
	\coordinate (4123) at (1/2, 5/6, 1/3);
	\coordinate (4132) at (1/2, 2/3, 1/6);
	\coordinate (3214) at (1/2, 5/6, 2/3);
	\coordinate (3124) at (1/3, 5/6, 1/2);
	\coordinate (2431) at (2/3, 1/6, 1/2);
	\coordinate (1432) at (1/2, 1/6, 2/3);
	\coordinate (1423)  at (1/2, 1/3, 5/6);
	\coordinate (1342)  at (1/3, 1/6, 1/2);
	\coordinate (2341) at (1/2, 1/6, 1/3);
	\coordinate (3241) at (1/2, 1/3, 1/6);
	\coordinate (1243) at (1/6, 1/3, 1/2);
	\coordinate (2143) at (1/6, 1/2, 1/3);
	\coordinate (1234) at (1/6, 1/2, 2/3);
	\coordinate (2134) at (1/6, 2/3, 1/2);
	
	\draw[thick, draw=blue!70] (4213)--(4312)--(3412)--(2413)--(2314)--(3214)--cycle;
	\draw[thick, draw=blue!70] (4312)--(4321)--(3421)--(3412);
	\draw[thick, draw=blue!70] (3421)--(2431)--(1432)--(1423)--(2413);
	\draw[thick, draw=blue!70] (1423)--(1324)--(2314);
	\draw[thick, draw=blue!70] (1432)--(1342)--(1243)--(1234)--(1324);
	\draw[thick, draw=blue!70] (1234)--(2134)--(3124)--(3214);
	\draw[thick, draw=blue!70] (3124)--(4123)--(4213);
	
	\draw[thick, draw=blue!70, dashed] (2134)--(2143)--(3142)--(4132)--(4123);
	\draw[thick, draw=blue!70, dashed] (2143)--(1243);
	\draw[thick, draw=blue!70, dashed] (3142)--(3241)--(2341)--(1342);
	\draw[thick, draw=blue!70, dashed] (2341)--(2431);
	\draw[thick, draw=blue!70, dashed] (3241)--(4231)--(4132);
	\draw[thick, draw=blue!70, dashed] (4231)--(4321);

\node[label = {[label distance = 0cm]right:2341}] at (2341) {};

	\draw[red line] (3214)--(2314)--(2413)--(3412)--cycle;
	\draw[red line] (2314)--(1324)--(1423)--(2413);
	\draw[red line] (1324)--(1342)--(1432)--(1423);
	\draw[red line] (1432)--(3412);
	\draw[red line] (1324)--(3124);
	\draw[red line](3124)--(3214);
	\draw[red line, dashed] (3124)--(3142)--(1342);
	\draw[red line, dashed] (3142)--(3412);
	
\node [label = {[label distance = 0cm]left:1234}] at (1234) {};
\node[label = {[label distance = 0cm]left:1243}] at (1243) {};
\node[label = {[label distance = 0cm]right:1324}, red node] at (1324) {};
\node[label = {[label distance = 0cm]left:1342}, red node] at (1342) {};
\node [label = {[label distance = 0cm]above:1423},red node] at (1423) {};
\node[label = {[label distance = -0.2cm]below:1432}, red node] at (1432) {};
\node [label = {[label distance = 0cm]left:2134}] at (2134) {};
\node[label = {[label distance = -0.1cm]below right:2143}] at (2143) {};
\node[label = {[label distance = 0cm]below:2314}, red node] at (2314) {};
\node[label = {[label distance = 0cm]left:2413}, red node] at (2413) {};
\node[label = {[label distance = -0.2cm]below:2431}] at (2431) {};
\node[label = {[label distance = -0.2cm]above:3124}, red node] at (3124) {};
\node[label = {[label distance = -0.2cm]above:3142}, red node] at (3142) {};
\node[label = {[label distance = -0.2cm]above:3214}, red node] at (3214) {};
\node [label = {[label distance = -0.1cm]above:3241}] at (3241) {};
\node[label = {[label distance = 0cm]below left:3412}, red node] at (3412) {};
\node[label = {[label distance = 0cm]right:3421}] at (3421) {};
\node[label = {[label distance = -0.2cm]above:4123}] at (4123) {};
\node [label = {[label distance = 0cm]below:4132}] at (4132) {};
\node[label = {[label distance = 0cm]right:4213}] at (4213) {};
\node[label = {[label distance = 0cm]left:4231}] at (4231) {};
\node[label = {[label distance = 0cm]right:4312}] at (4312) {};
\node [label = {[label distance = 0cm]right:4321}] at (4321) {};
	\end{tikzpicture}
\subcaption{Bruhat interval polytope $Q_{1324,3412}$.} \label{fig_BIP_1324-3412}
\end{subfigure}
\caption{Permutohedron $\textrm{Perm}_3$ and an example of a Bruhat interval polytope.}\label{fig:1324-3412}
\end{figure}

\section{Relation with Richardson varieties}\label{sec:Relationship with Richardson varieties}

In this section, we review the relation between Bruhat interval polytopes and Richardson varieties, and introduce the connection between combinatorial properties of Bruhat interval polytopes and geometric properties of Richardson varieties.

Let $G = \GL_n(\mathbb{C})$, $B \subset G$ the set of upper triangular matrices, and $\mathbb{T}\subset G$ the set of diagonal matrices. Let $B^- \subset G$ be the set of lower triangular matrices. Then $\mathbb{T} := B \cap B^-$ and $B^-=w_0Bw_0$.
The manifold $G/B$ can be identified with the \emph{flag variety} $\Fl_n$ which is defined to be
\[
\Fl_n := \{ (\{0\} \subset V_1 \subset V_2 \subset \cdots \subset V_n = \C^n) \mid \dim_{\C} V_i = i \quad \text{ for all }i = 1,\dots,n\}.
\]

For an element $w \in \mathfrak{S}_n$, we define the permutation matrix $\begin{bmatrix}\ve_{w(1)} & \cdots & \ve_{w(n)}\end{bmatrix} \in \GL_n(\C)$ where $\ve_1,\dots,\ve_n$ are the standard basis vectors in $\mathbb{R}^n$. We will write it simply $w$ if there is no confusion.
For an element $w \in \mathfrak{S}_n$, we denote the \emph{Schubert variety} $\overline{BwB/B}$ (respectively, the \emph{opposite Schubert variety} $\overline{B^-wB/B}$) in the flag variety $G/B$ by $X_w$ (respectively, $X^w$). The left multiplication by~$\mathbb{T}$ on $G$ induces the $\mathbb{T}$-action on $G/B$ which leaves both $X_w$ and $X^w$ invariant. The set of $\mathbb{T}$-fixed points in $G/B$ bijectively corresponds to the symmetric group $\mathfrak{S}_n$ through the correspondence $u\in \mathfrak{S}_n \to uB \in G/B$. A fixed point $uB$ is contained in $X_w$ if and only if $u \leq w$ in Bruhat order and $uB$ is contained in $X^w$ if and only if $u \geq w$ in Bruhat order (see \cite[\S 10.5]{ful97}).

For elements $v$ and $w \in \mathfrak{S}_{n}$ with $v\leq w$, we define the \emph{Richardson variety}
$X^{v}_w$ by $X^v \cap X_w$. Then $X^e_{w_0}=G/B$, $X^e_{w}=X_w$, and $X^w_{w_0}=X^w$. Furthermore, $X^v_w$ is also $\mathbb{T}$-invariant and the $\mathbb{T}$-fixed points of $X^v_w$ correspond to the elements in the Bruhat interval $[v,w]$. It is known that
\begin{equation}\label{eq:dim_schubert}
\dim_\C X^v_w =\ell(w)-\ell(v),
\end{equation}
see~\cite{Brion_Lecture}.

The full flag variety $G/B$ has the symplectic form $\omega_{\lambda}$ due to Kirillov, Kostant, and Souriau for a regular dominant weight $\lambda$. When we choose the weight $\lambda$ as the sum of all fundamental weights, the permutohedron~$\textrm{Perm}_{n-1}$ is the moment map image of~$G/B$ (see, for example,~\cite[Corollary IV.4.11]{AudinTorus} and references therein).
Now we describe the moment map $\mu\colon G/B\to \R^n$ explicitly using the Pl\"{u}cker coordinates. We define the set
\[
I_{d,n} = \{\underline{\mathbf{i}} = (i_1,\dots,i_d) \in \Z^d \mid 1 \leq i_1 < \cdots < i_d \leq n \}.
\]
For an element $x = (x_{ij}) \in G = \GL_n(\C)$, the $\underline{\mathbf{i}}$th Pl\"{u}cker coordinate $p_{\underline{\mathbf{i}}}(x)$ of~$x$ is given by the $d \times d$ minor of~$x$, with row indices $i_1,\dots,i_d$ and the column indices $1,\dots,d$ for $\underline{\mathbf{i}} = (i_1,\dots,i_d) \in I_{d,n}$. Then the Pl\"{u}cker embedding is defined to be
\begin{equation}\label{eq_Plucker_embedding}
\psi \colon G/B \to \prod_{d=1}^{n-1} \mathbb{C}P^{{n\choose d}-1},
\quad xB \mapsto \prod_{d=1}^{n-1} (p_{\underline{\mathbf{i}}}(x))_{\underline{\mathbf{i}} \in I_{d,n}}.
\end{equation}
The map $\psi$ is $\mathbb{T}$-equivariant with respect to the action of
$\mathbb{T}$ on $\prod_{d=1}^{n-1} \mathbb{C} P^{{n\choose d}-1}$ given by
\[
(t_1,\dots,t_n) \cdot (p_{\underline{\mathbf{i}}})_{\underline{\mathbf{i}} \in I_{d,n}}
:= (t_{i_1}\cdots t_{i_d} \cdot p_{\underline{\mathbf i}})_{\underline{\mathbf{i}} \in I_{d,n}}
\]
for $(t_1,\dots,t_n) \in \mathbb{T}$ and $\underline{\mathbf{i}} = (i_1,\dots,i_d)$. Then the moment map $\tilde\mu\colon \prod_{d=1}^{n-1}\CP^{{n\choose d}-1} \to \R^n$ is given by
\begin{equation}\label{eq:moment-map}
(p_{\underline{\mathbf{i}}})_{\underline{\mathbf{i}}\in I_{d,n}} \mapsto -\sum_{d=1}^{n-1}\left\{\frac{1}{\sum_{\underline{\mathbf{i}}\in I_{d,n}} | p_{\underline{\mathbf{i}}}|^2}\left(\sum_{1\in \underline{\mathbf{i}}\in I_{d,n}}|p_{\underline{\mathbf{i}}}|^2,\ldots,\sum_{n\in \underline{\mathbf{i}}\in I_{d,n}}|p_{\underline{\mathbf{i}}}|^2\right)\right\} + \mathbf{c},
\end{equation}  where $\mathbf{c}$ is a constant vector. By setting $\mathbf{c}=(n,\ldots,n)$ in~\eqref{eq:moment-map} and $\mu := \tilde{\mu} \circ \psi$, we can see the following.
\begin{lemma}\label{lem:moment_map}
	The moment map $\mu$ sends the fixed point $uB \in G/B$ to $(u^{-1}(1),\dots,u^{-1}(n)) \in \mathbb{R}^n$.
\end{lemma}
\begin{proof}
	For a permutation $u \in \mathfrak{S}_n$, the Pl\"{u}cker coordinates $(p_{\underline{\mathbf{i}}})_{\underline{\mathbf{i}} \in I_{d,n}}$ of $uB$ are given as follows:
		\[
		p_{\underline{\mathbf{i}}} = \begin{cases} 1 & \text{ if }\underline{\mathbf{i}} = \{u(1),\dots,u(d)\}\!\uparrow,\\
		0 & \text{ otherwise,}
		\end{cases}
		\]
		for each $\underline{\mathbf{i}} \in I_{d,n}.$ Here, for a subset $S\subset [n]$, we denote by $S\!\uparrow$ the ordered tuple obtained from $S$ by sorting its elements in ascending order.
	Therefore one can see that for a fixed $d \in [n-1]$, $\sum_{\underline{\mathbf{i}}\in I_{d,n}} |p_{\underline{\mathbf{i}}}|^2=1$ and the vector
	\[
	\left( \sum_{1 \in \underline{\mathbf{i}}\in I_{d,n}} |p_{\underline{\mathbf{i}}}|^2, \dots, \sum_{n \in \underline{\mathbf{i}}\in I_{d,n}} |p_{\underline{\mathbf{i}}}|^2 \right)
	\]
	becomes an integer vector whose entries are $1$ for coordinates in $\{u(1),\dots,u(d)\}$ and $0$ otherwise.
	Hence the summation
	\[
	-\sum_{d=1}^{n-1} 
	\left\{ \frac{1}{\sum_{\underline{\mathbf{i}}\in I_{d,n}} | p_{\underline{\mathbf{i}}}|^2}\left( \sum_{1 \in \underline{\mathbf{i}}\in I_{d,n}} |p_{\underline{\mathbf{i}}}|^2, \dots, \sum_{n \in \underline{\mathbf{i}}\in I_{d,n}} |p_{\underline{\mathbf{i}}}|^2 \right)
	\right\}
	\]
	is an integer vector such that the $u(k)$-entry is $-(n-k)$. Therefore, the moment map image $\mu(uB)$ is an integer vector whose $u(k)$-entry is $k$ since $\mathbf{c}=(n,\cdots,n)$ in~\eqref{eq_Plucker_embedding}. This implies that $\mu(uB) = (u^{-1}(1),\dots,u^{-1}(n))$ since $u^{-1}(u(k)) = k$ for all $k$.
\end{proof}

It follows from Lemma~\ref{lem:moment_map} that for $v$ and $w \in\mathfrak{S}_n$ with $v\leq w$, the Bruhat interval polytope $Q_{v,w}$ is the moment map image of the Richardson variety $X^{v^{-1}}_{w^{-1}}.$

\begin{example}\label{example_Plucker_GL3}
Suppose that $G = \GL_3(\mathbb{C})$. Then the Pl\"{u}cker embedding
$\psi \colon {G/B} \to \C {P}^{{3\choose 1}-1} \times \C {P}^{{3\choose 2}-1}$ maps an element $x = (x_{ij}) \in \GL_3(\mathbb{C})$
to
\begin{align*}
&([p_1(x),p_2(x),p_3(x)], [p_{1,2}(x), p_{1,3}(x), p_{2,3}(x)]) \\
&\qquad = ([x_{11}, x_{21}, x_{31}], [x_{11} x_{22} - x_{21}x_{12}, x_{11}x_{32} - x_{31} x_{12}, x_{21}x_{32} - x_{31}x_{22}]).
\end{align*}
{Since the action of $\mathbb{T}$ on $\GL_3(\C)$ is given by}
\begin{align*}
(t_1,t_2,t_3) \cdot \begin{pmatrix}
x_{11} & x_{12} & x_{13} \\
x_{21} & x_{22} & x_{23} \\
x_{31} & x_{32} & x_{33}
\end{pmatrix}
= \begin{pmatrix}
t_1 x_{11} & t_1 x_{12} & t_1 x_{13} \\
t_2 x_{21} & t_2 x_{22} & t_2 x_{23} \\
t_3 x_{31} & t_3 x_{32} & t_3 x_{33}
\end{pmatrix},
\end{align*}
{one can easily check that the map $\psi$ is $\mathbb{T}$-equivariant.} The moment map $\tilde\mu\colon\CP^{{3\choose 1}-1}\times \CP^{{3\choose 2}-1}\to \R^3$ is given by
\begin{equation*}
\begin{split}
&([p_1,p_2,p_3],[p_{12},p_{13},p_{23}])\\&\qquad\mapsto -\frac{1}{|p_1|^2+|p_2|^2+|p_3|^2}\left({|p_1|^2},{|p_2|^2},{|p_3|^2}\right)\\&\qquad\quad -\frac{1}{|p_{12}|^2+|p_{13}|^2+|p_{23}|^2}\left({|p_{12}|^2+|p_{13}|^2},{|p_{12}|^2+|p_{23}|^2},{|p_{13}|^2+|p_{23}|^2}\right)\\
&\qquad\quad+(3,3,3).
\end{split}
\end{equation*} Then one can see that \[\mu(312B)=\tilde\mu\circ\psi(312B)=\tilde\mu(([0,0,1],[0,1,0]))=-(0,0,1)-(1,0,1)+(3,3,3)=(2,3,1).
\]
\end{example}

We call a $\mathbb{T}$-orbit in $X^v_w$ \emph{generic} if its closure contains all the $\mathbb{T}$-fixed points in $X^v_w$ and call a point in $X^v_w$ \emph{generic} if it is in a generic $\mathbb{T}$-orbit. We will denote the closure of a generic $\mathbb{T}$-orbit in $X^v_w$ by $Y^v_w$. Then $Y^v_w$ is the projective toric variety defined by the polytope $Q_{v^{-1},w^{-1}}$. Hence it follows from~\eqref{eq:dim_schubert} that
\begin{equation}\label{eq_dim_Qvw_and_length}
\dim Q_{v,w}=\dim_\C Y^{v^{-1}}_{w^{-1}}\leq \dim_\C X^{v^{-1}}_{w^{-1}}=\ell(w^{-1})-\ell(v^{-1})=\ell(w)-\ell(v).
\end{equation}
Motivated by this observation, we introduce the following terminology.
\begin{definition}
	The Bruhat interval polytope $Q_{v,w}$ is called \emph{toric} if $\dim Q_{v,w}=\ell(w)-\ell(v)$.
\end{definition}

Equation~\eqref{eq_dim_Qvw_and_length} implies that the Richardson variety $X^v_w$ is a toric variety, that is, $X^v_w=Y^v_w$, if and only if the Bruhat interval polytope $Q_{v^{-1},w^{-1}}$ is toric.
In general, Bruhat interval polytopes $Q_{v,w}$ and $Q_{v^{-1},w^{-1}}$ are not combinatorially equivalent even though  $x \lessdot y$ is a cover in $[v,w]$ if and only if $x^{-1}\lessdot y^{-1}$ is a cover in $[v^{-1},w^{-1}]$ (see~Remark~\ref{rmk_BIP_inverses_are_not_combinatorially_equivalent}).
But they have the same dimension. 
We get the following whose proof will be given in the next section.
\begin{proposition}\label{lem:dim}
	Two Bruhat interval polytopes $Q_{v,w}$ and $Q_{v^{-1},w^{-1}}$ have the same dimension. In particular, $X^v_w$ is a toric variety if and only if so is $X^{v^{-1}}_{w^{-1}}$. Indeed, the Richardson variety $X^v_w$ is a toric variety if and only if $Q_{v,w}$ is toric.
\end{proposition}

\section{Properties of Bruhat interval polytopes}\label{sec:Properties of Bruhat interval polytopes}
In this section, we review some notations and facts about Bruhat interval polytopes and related graphs introduced in~\cite{le-ma18} and~\cite{ts-wi15}. Then we interpret  combinatorial properties of Bruhat interval polytopes using these graphs.
Using this interpretation, we provide a proof of Proposition~\ref{lem:dim} which shows that Bruhat interval polytopes $Q_{v,w}$ and $Q_{v^{-1}, w^{-1}}$ have the same dimension.

We first set up notations and terminologies related to digraphs (or directed graphs). A \emph{digraph} is an ordered pair $G = (V(G), E(G))$, where
\begin{itemize}
\item $V(G)$ is a set whose elements are called \emph{vertices}, and
\item $E(G)$ is a multiset of ordered pairs of vertices, called \emph{directed edges}.
\end{itemize}
For two vertices $i$ and $j$ of a given graph $G$, $i$ can \emph{reach}  $j$ if there is a (directed) path from $i$ to $j$.
A digraph $G$ is said to be \emph{acyclic} if there is no directed cycle.
The underlying graph of~$G$ is the undirected graph created using all of the vertices in $V(G)$ and replacing all directed edges in $E(G)$ with undirected edges.
A digraph is connected (or weakly connected) if the underlying graph is a connected graph.  Hence if $i$ can reach $j$, then $i$ and $j$ are connected, but the converse is not true in general.
If $V(G)=[n]$, then we can define $B(G)$ to be a partition of the set $[n]$ such that each block corresponds to a connected component of~$G$.

Let $v,w\in\mathfrak{S}_n$ with $v\leq w$. For $u\in [v,w]$, we define the following two sets:
\begin{align*}
&\overline{T}(u,[v,w])=\{t\in T\mid \exists\, z \stackrel{t}{\gtrdot}u,\,z\in [v,w]\}=\{t\in T\mid u\lessdot ut\leq w\} \text{ and }\\
&\underline{T}(u,[v,w])=\{t\in T\mid \exists\, z \stackrel{t}{\lessdot}u,\,z\in [v,w]\}=\{t\in T\mid v\leq ut\lessdot u\}.
\end{align*}
Here, $T$ is the set of transpositions (see~\eqref{eq_def_of_T}). The digraph $G_u^{v,w}$ is defined as follows:
\begin{enumerate}
	\item The vertices of $G_{u}^{v,w}$ are $\{1,2,\ldots,n\}$.
	\item There is a directed edge from $i$ to $j$ for every $(i,j)\in\overline{T}(u,[v,w])$.
	\item There is a directed edge from $j$ to $i$ for every $(i,j)\in \underline{T}(u,[v,w])$.
\end{enumerate}  Then the dimension of the Bruhat interval polytope $Q_{v,w}$ is determined by the number of blocks of the partition determined by the graph $G_{v}^{v,w}$ or $G_w^{v,w}$, that is,
\begin{equation}\label{eq:dim-BIP}
\dim Q_{v,w}=n-\# B(G_{v}^{v,w})=n-\#B(G_{w}^{v,w}),
\end{equation}
see Theorem~{4.6} and Proposition~{4.10} of~\cite{ts-wi15}.

\begin{example}\label{eg:atom-graph}
  Let $[v,w]=[1324,4231]$. Then the digraph $G_{v}^{v,w}$ is connected (see Figure~\ref{fig:graph_G_at}) and the Bruhat interval polytope $Q_{v,w}$ is of dimension~3.
  
  \begin{figure}[h]
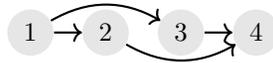

		\centering
		\tikz{
			\node[fill=gray!20] (1) at (0,0) [circle] {$1$};
			\node[fill=gray!20] (2) at (1,0) [circle] {$2$};
			\node[fill=gray!20] (3) at (2,0) [circle] {$3$};
			\node[fill=gray!20] (4) at (3,0) [circle] {$4$};
			\draw  (1) edge[->,thick] (2)		(2) edge[->,thick, bend right] (4) (1) edge[->,thick, bend left] (3) (3) edge [->,thick] (4);
		}
		\caption{Graph $G_{1324}^{1324,4231}$.}
		\label{fig:graph_G_at}
	\end{figure}
\end{example}

Now we give the proof of Proposition~\ref{lem:dim} which claims that two Bruhat interval polytope $Q_{v,w}$ and $Q_{v^{-1},w^{-1}}$ have the same dimension.

\begin{proof}[Proof of Proposition~\ref{lem:dim}]
Let $\overline{T}(v,[v,w])=\{(i_1,j_1),\,(i_2,j_2),\,\ldots,\,(i_k,j_k)\}$. Since $\underline{T}(v,[v,w])=\emptyset$, we get $E(G_{v}^{v,w})=\overline{T}(v,[v,w])$. Since $(i,j)v^{-1}=v^{-1}(v(i),v(j))$ for every transposition $(i,j)\in T$, we have that
$$E(G_{v^{-1}}^{v^{-1},w^{-1}})=\overline{T}(v^{-1},[v^{-1},w^{-1}])=\{(v(i_1),v(j_1)),\,(v(i_2),v(j_2)),\,\ldots,\,(v(i_k),v(j_k))\}.$$ Since $v$ is a bijection on $[n]$, the graph $G_{v}^{v,w}$ is isomorphic to $G_{v^{-1}}^{v^{-1},w^{-1}}$, and hence the partitions $B(G_{v}^{v,w})$ and $B(G_{v^{-1}}^{v^{-1},w^{-1}})$ consist of the same number of blocks. Therefore, $\dim Q_{v,w}=\dim Q_{v^{-1},w^{-1}}$ by~\eqref{eq:dim-BIP}.
\end{proof}

Even though $Q_{v,w}$ and $Q_{v^{-1}, w^{-1}}$ can have the same dimension, their face structures cannot be compared with each other in general (see~Remark~\ref{rmk_BIP_inverses_are_not_combinatorially_equivalent}). Now, we define a digraph $G_{x,y}^{v,w}$ for each $[x,y]\subset [v,w]$ and use it to verify that $Q_{x,y}$ is a face of $Q_{v,w}$.
The digraph $G_{x,y}^{v,w}$ is defined as follows:
\begin{enumerate}
	\item The vertices of $G^{v,w}_{x,y}$ are $\{1,2,\ldots,n\}$, with nodes $i$ and $j$ identified if they are in the same block of~$B(G_x^{x,y})$.
	\item There is a directed edge from $i$ to $j$ for every $(i,j)\in\overline{T}(y,[v,w])$.
	\item There is a directed edge from $j$ to $i$ for every $(i,j)\in \underline{T}(x,[v,w])$.
\end{enumerate}
It has been known from~\cite[Theorem 4.1]{ts-wi15} that every face of a Bruhat interval polytope is itself a Bruhat interval polytope.
 Moreover, for $[x,y] \subset [v,w]$, one can determine whether the Bruhat interval polytope $Q_{x,y}$ is a face of $Q_{v,w}$ by checking the acyclicity of $G_{x,y}^{v,w}$.

\begin{theorem}[{\cite[Theorem 4.19]{ts-wi15}}]\label{def:face-graph}
For $[x,y]\subset [v,w]$, the  Bruhat interval polytope $Q_{x,y}$ is a face of the Bruhat interval polytope $Q_{v,w}$ if and only if the graph $G^{v,w}_{x,y}$ is an acyclic digraph.
\end{theorem}

\begin{example}\label{example:inverse-acyclic}
  Let $[v,w]=[1324,4231]$ and consider $[x,y]=[1432,2431] \subset [v,w]$.  Note that $[v^{-1},w^{-1}]=[v,w]$ and $[x^{-1},y^{-1}]=[1432,4132]$. Since $y=x(1,4)$, $B(G_{x}^{x,y})=14|2|3$ and $B(G_{x^{-1}}^{x^{-1},y^{-1}})=12|3|4$. We have that
  \[
  \overline{T}(y,[v,w])=\{(1,2)\} \quad \text{ and } \quad \underline{T}(x,[v,w])=\{(2,3),(3,4)\},
  \]
  and the graph $G^{v,w}_{x,y}$ is acyclic (see Figure~\ref{fig_acyclic_1324_4231}). On the other hand, we get
  \[
  \overline{T}(y^{-1},[v^{-1},w^{-1}])=\{(2,4)\} \quad \text{ and } \quad \underline{T}(x^{-1},[v^{-1},w^{-1}])=\{(2,3),(3,4)\},
  \]
  so that the graph $G^{v^{-1},w^{-1}}_{x^{-1},y^{-1}}$ is a directed cycle (see Figure~\ref{fig_cycle_1324_4231}).
 \begin{figure}[h!]
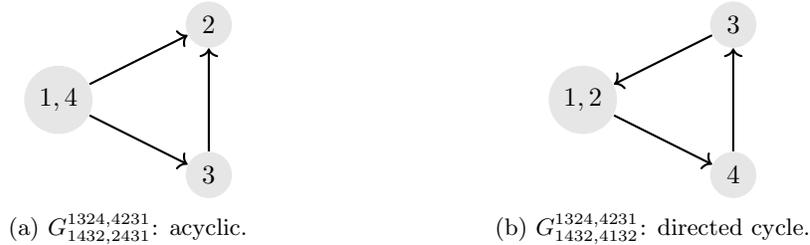

 \centering
 \begin{subfigure}[b]{.45\textwidth}
 \centering
\tikz{
\node[fill=gray!20] (a) at (0,0) [circle] {$1,4$};
\node[fill=gray!20] (b) at (2,1) [circle] {$2$};
\node[fill=gray!20] (c) at (2,-1) [circle] {$3$};
\draw (a) edge[->,thick] (b) (a) edge[->,thick] (c)
(c) edge[->,thick] (b);
}
\subcaption{$G^{1324,4231}_{1432,2431}$: acyclic.}
\label{fig_acyclic_1324_4231}
 \end{subfigure}
\begin{subfigure}[b]{.45\textwidth}
\centering
 \tikz{
\node[fill=gray!20] (a) at (0,0) [circle] {$1,2$};
\node[fill=gray!20] (b) at (2,1) [circle] {$3$};
\node[fill=gray!20] (c) at (2,-1) [circle] {$4$};
\draw (b) edge[->,thick] (a) (a) edge[->,thick] (c)
(c) edge[->,thick] (b);
}
\caption{$G^{1324,4231}_{1432,4132}$: directed cycle.}
\label{fig_cycle_1324_4231}
 \end{subfigure}
 \caption{$G^{v,w}_{x,y}$ is acyclic but $G^{v^{-1},w^{-1}}_{x^{-1},y^{-1}}$ is not.}\label{fig:graph}
 \end{figure}\label{fig:graph-face}
\end{example}

\begin{remark}\label{rmk_BIP_inverses_are_not_combinatorially_equivalent}
	The graph $G_{1432,2431}^{1324,4231}$ is acyclic but the graph $G_{1432,4132}^{1324,4231}$ is not as in~Example~\ref{example:inverse-acyclic}. 	
	Therefore, the fact that $Q_{x,y}$ is a face of $Q_{v,w}$ does not imply that $Q_{x^{-1},y^{-1}}$ is a face of $Q_{v^{-1},w^{-1}}$. See Figure~\ref{fig_BIP_1324-3412} for the Bruhat interval polytope $Q_{1324,3412}$.
	Moreover, the Bruhat interval polytopes $Q_{v,w}$ and $Q_{v^{-1},w^{-1}}$ are not combinatorially equivalent in general. For example, one can check that two Bruhat interval polytopes $Q_{12345,35412}$ and $Q_{12345,45132}$ are not combinatorially equivalent using a computer program, for example, using SAGE.
\end{remark}

A \emph{transitive reduction} of a digraph $G$ is another digraph with the same vertices and as few edges as possible, such that if there is a directed path from vertex $i$ to vertex $j$, then there is also such a path in the reduction. That is, the reduction has the same reachability relations as $G$. Remarkably, the transitive reduction of a finite acyclic digraph is unique and is a subgraph of the given graph. We can find the transitive reduction of a finite acyclic digraph by removing each directed edge $i\to j$ if there is a directed path from $i$ to $j$. See~\cite{a-ga-ul72} for more details.

\begin{example} A transitive reduction of the graph in Figure~\ref{fig:graph_G_at} is itself. Now consider the graphs in Example~\ref{fig:graph-face}. For the graph in Figure~\ref{fig_acyclic_1324_4231}, the node $\{1,4\}$ reaches to the node $\{2\}$ via two different ways: $\{1,4\} \to \{2\}$ and $\{1,4\} \to \{3\} \to \{2\}$. Since this graph is a finite acyclic digraph, it has a unique transitive reduction (see Figure~\ref{fig:transitive_reduction_1}).
On the other hand, the graph $G_{1432,4132}^{1324,4231}$ in Figure~\ref{fig_cycle_1324_4231} and the graph given in Figure~\ref{fig:transitive_reduction_2} have the same reachability relations, and those graphs are a transitive reduction of each other.\footnote{For this reason, uniqueness of a transitive reduction fails for digraphs with cycles.}
	 \begin{figure}[h!]
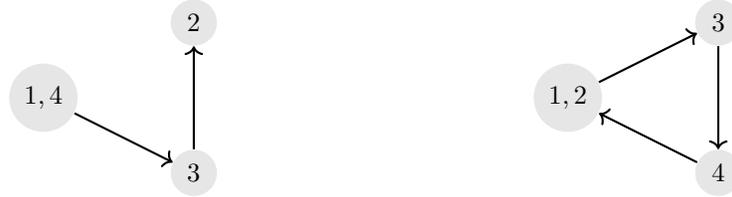

		\centering
		\begin{subfigure}[b]{.45\textwidth}
			\centering
			\tikz{
				\node[fill=gray!20] (a) at (0,0) [circle] {$1,4$};
				\node[fill=gray!20] (b) at (2,1) [circle] {$2$};
				\node[fill=gray!20] (c) at (2,-1) [circle] {$3$};
				\draw (a) edge[->,thick] (c)
				(c) edge[->,thick] (b);
			}
			\subcaption{The transitive reduction of Figure~\ref{fig_acyclic_1324_4231}.}
			\label{fig:transitive_reduction_1}
		\end{subfigure}
		\begin{subfigure}[b]{.45\textwidth}
			\centering
			\tikz{
				\node[fill=gray!20] (a) at (0,0) [circle] {$1,2$};
\node[fill=gray!20] (b) at (2,1) [circle] {$3$};
\node[fill=gray!20] (c) at (2,-1) [circle] {$4$};
\draw (b) edge[<-,thick] (a) (a) edge[<-,thick] (c)
(c) edge[<-,thick] (b);
			}
			\caption{A transitive reduction of Figure~\ref{fig_cycle_1324_4231}.}
			\label{fig:transitive_reduction_2}
		\end{subfigure}
		\caption{Transitive reductions.}\label{fig:transitive_reduction}
	\end{figure}
\end{example}

Since every $u \in [v,w]$ can be realized as a vertex of the polytope $Q_{v,w}$, the graph $G^{v,w}_{u}$ is an acyclic digraph by Theorem~\ref{def:face-graph}. Hence $G^{v,w}_{u}$ has a unique transitive reduction and we denote it by $\widetilde{G}^{v,w}_u$.\footnote{Note that the graph $\widetilde{G}^{e,w^{-1}}_{u^{-1}}$ is equal to the graph $\Gamma_w(u)$ in~\cite{le-ma18} for $u \leq w$.} Note that a directed edge from $i$ to $j$ in $G^{v,w}_u$ is disappeared in $\widetilde{G}^{v,w}_u$ if and only if the vector $\ve_i-\ve_j$ can be expressed as a sum of vectors in the set $\{\ve_k-\ve_\ell\mid (k,\ell)\in E(\widetilde{G}^{v,w}_u)\}.$ Hence $Q_{v,w}$ is contained in 
\begin{equation}\label{eq:cone}
(u(1),\ldots,u(n))+\mathrm{Cone}(\{\ve_i-\ve_j\mid (i,j)\in E(\widetilde{G}^{v,w}_u)\}).
\end{equation}

\begin{corollary}[{see~\cite[Corollary 4.20]{ts-wi15}}]\label{cor:edge_vectors}
Let $v,w\in\mathfrak{S}_n$ with $v\leq w$. For $u\in [v,w]$, the number of edges meeting at the vertex~$u$ in $Q_{v,w}$ equals the number of edges in $\widetilde{G}^{v,w}_u$. Furthermore, the primitive direction vectors of the edge emanating from a vertex $u$ of $Q_{v,w}$ are given by $\ve_i-\ve_j$ for a directed edge from $i$ to $j$ in the graph~$\widetilde{G}^{v,w}_u$.
\end{corollary}

We can check the smoothness of a vertex of a Bruhat interval polytope by using the graph $\widetilde{G}^{v,w}_u$.

\begin{proposition}\label{prop:smooth}
Let $v$ and $w$ be in $\mathfrak{S}_n$ with $v\leq w$ and let $u\in [v,w]$. The following are equivalent:
\begin{enumerate}
\item The vertex $u$ is simple in $Q_{v,w}$.
\item The vertex $u$ is smooth in $Q_{v,w}$.
\item The undirected underlying graph of $\widetilde{G}^{v,w}_u$ is a forest.
\end{enumerate}
\end{proposition}
\begin{proof}
The proof of \cite[Lemma 6.6]{le-ma18} works here too.
\end{proof}

The following proposed conjecture is a generalized version of Conjecture 7.16 in~\cite{le-ma18}.
\begin{conjecture}\label{conj}
   The polytope $Q_{v,w}$ is a simple polytope if and only if the vertices $v$ and $w$ are simple. Equivalently, if two graphs $\widetilde{G}^{v,w}_v$ and $\widetilde{G}^{v,w}_v$ are forests, then $\widetilde{G}^{v,w}_u$ is a forest for every $v\leq u\leq w$.
\end{conjecture}
In fact, we have computer-based evidence \cite{tsuchiya} that the conjecture above is true for $n\leq 6$.

Note that the dimension of the cone in~\eqref{eq:cone} is independent of the choice of $u$ since it is equal to the dimension of the polytope $Q_{v,w}$. This says that $\#B(\widetilde{G}_{u}^{v,w})$ is independent of the choice of~$u$.

\begin{corollary}
Let $v,w\in\mathfrak{S}_n$ and $v\leq w$ and $u\in [v,w]$. The number of the connected components of the graph $\widetilde{G}^{v,w}_u$ is independent of~$u$, and $\dim Q_{v,w}=n-\#B(G^{v,w}_u)$.
\end{corollary}
\begin{proof}
Recall that the incidence matrix of a connected digraph on $[n]$ has rank~$(n-1)$. Let $M_u$ be the incidence matrix of the graph $\widetilde{G}^{v,w}_u$. Then the rank of $M_u$ is equal to the dimension of the cone in~\eqref{eq:cone}, and hence it is equal to the dimension of $Q_{v,w}$. Therefore, $$\dim Q_{v,w}=\rank(M_u)=n-\#B(\widetilde{G}^{v,w}_u).$$
Since $B(\widetilde{G}^{v,w}_u)=B({G}^{v,w}_u)$, this proves the corollary.
\end{proof}

We can further prove that the partition~$B(\widetilde{G}^{v,w}_u)$ is independent of~$u$ (see Proposition~\ref{lem:connected_component}). To give a proof of this, we introduce an operation on partitions.

For partitions $P$ and $Q$ of~$[n]$, we define a partition $P*Q$ of~$[n]$ as follows: two elements $i,j\in [n]$ are in a same block of $P*Q$ if and only if there is a sequence $i=i_1, i_2, \dots, i_k=j$ such that each consecutive pair $(i_\ell,i_{\ell+1})$ $(\ell=1,2,\dots,k-1)$ is in a same block of either $P$ or $Q$. 

\begin{example}
	Let
	\[
	\begin{split}
	P&=\{ \{1\}, \{2,3,4\}, \{5\}, \{6,7\}, \{8\}, \{9,10\}\},\\
	Q&=\{ \{1,3\}, \{2,4\}, \{5,7\}, \{6,8\}, \{9,10\}\}
	\end{split}
	\]
	be partitions on~$[10]$. Then
	\[
	P*Q=\{ \{1,2,3,4\}, \{5,6,7,8\}, \{9,10\}\}.
	\]
\end{example}

Note that if $G$ and $H$ are (undirected) graphs on $[n]$, then $B(G)\ast B(H)$ coincides with the partition determined by the graph sum of $G$ and $H$, the graph with adjacency matrix given by the sum of adjacency matrices of $G$ and $H$.

We introduce another way to compute the dimension of a Bruhat interval polytope by using a graph $G^{\mathcal{C}}$. Let $v,w\in\mathfrak{S}_n$ with $v\leq w$, and let $\mathcal{C}\colon v=x_{(0)}\lessdot x_{(1)}\lessdot \cdots\lessdot x_{(\ell)}=w$ be a maximal chain from $v$ to $w$. Then recall from~\cite[\S 4.2]{ts-wi15} that the undirected graph $G^{\mathcal{C}}$ on~$[n]$ is defined to be the graph whose edge set is the set of unordered pairs
\[
\{\{a,b\} \mid (a,b)=x_{(i)}^{-1}x_{(i+1)}  \in T \quad\text{ for some }0\leq i\leq \ell-1\}.
\]
Here, $T$ is the set of transpositions (see~\eqref{eq_def_of_T}).
Note that $G^\mathcal{C}$ can have multiple edges. It was shown in \cite[Corollary 4.8]{ts-wi15} that $B(G^{\mathcal{C}})$ is independent of the choice of~$\mathcal{C}$, and denoted  by $B_{v,w}$.\label{partition} 
 Furthermore, the dimension of the polytope~$Q_{v,w}$ is determined by the partition $B_{v,w}$, 
	\begin{equation}\label{eq_dim_Qvw_Bvw}
	\dim Q_{v,w} = n-\#B_{v,w},
	\end{equation}
 see~\cite[Theorem 4.6]{ts-wi15}.
 It was also shown in~\cite[Corollary 4.10]{ts-wi15} that 
\begin{equation}\label{eq_partitions_Bvw}
B(G^{v,w}_v)=B(G^{v,w}_w)=B_{v,w}.
\end{equation}

\begin{example}
	Suppose that $v = 1324 = s_2$ and $w = 4231 = s_3s_2s_1s_2s_3$. Choose a maximal chain $\mathcal{C}$ from $v$ to $w$:
	\begin{equation}\label{eq:chain}
	\begin{tikzcd}[column sep = 0cm, row sep = 0.2cm]
	\mathcal{C}: &s_2 &\lessdot &s_3s_2 &\lessdot &s_3s_2s_1 &\lessdot &s_3s_2s_1s_2 &\lessdot& s_3s_2s_1s_2s_3. \\
	& x_{(0)} \arrow[u, equal] & & x_{(1)} \arrow[u, equal] & & x_{(2)}  \arrow[u, equal] & & x_{(3)} \arrow[u, equal] & & x_{(4)} \arrow[u, equal]
	\end{tikzcd}
	\end{equation}
	Then we have that
	\[
	x_{(0)}^{-1}x_{(1)} = s_2s_3s_2 = (2,4), \quad
	x_{(1)}^{-1}x_{(2)} = s_1  = (1,2), \quad
	x_{(2)}^{-1}x_{(3)} = s_2 = (2,3), \quad
	x_{(3)}^{-1}x_{(4)} = s_3 = (3,4).
	\]
	Hence the corresponding graph is given as in Figure~\ref{fig_graph_Bvw} and the partition $B_{v,w}$ is $\{[4]\}$.  Therefore, the Bruhat interval polytope $Q_{1324,4231}$ is of dimension~$3$.
	\begin{figure}[h]
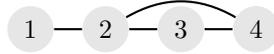

		\centering
		\tikz{
			\node[fill=gray!20] (1) at (0,0) [circle] {$1$};
			\node[fill=gray!20] (2) at (1,0) [circle] {$2$};
			\node[fill=gray!20] (3) at (2,0) [circle] {$3$};
			\node[fill=gray!20] (4) at (3,0) [circle] {$4$};
			\draw  (1) edge[thick] (2)		(2) edge[thick] (3) (3) edge[thick] (4) (2) edge [thick, bend left] (4);
		}
		\caption{Graph $G^{\mathcal{C}}$ for the chain $\mathcal{C}$ in~\eqref{eq:chain}.}
		\label{fig_graph_Bvw}
	\end{figure}
\end{example}

\begin{proposition}\label{lem:connected_component}
Let $v,w\in\mathfrak{S}_n$ and $v\leq w$ and $u\in [v,w]$. The connected components of the graph $\widetilde{G}^{v,w}_u$ are independent of~$u$. Furthermore, $B(\widetilde{G}^{v,w}_u)=B_{v,w}$.
\end{proposition}
\begin{proof}
Note that $B(G^{v,w}_u)=B(\widetilde{G}^{v,w}_u)$. Then we can see that
\begin{equation*}
\begin{array}{rll}
B(\widetilde{G}^{v,w}_u)&=B(\widetilde{G}^{v,u}_u)\ast B(\widetilde{G}^{u,w}_u)&\\
&=B_{v,u}\ast B_{u,w}&\text{ (by~\eqref{eq_partitions_Bvw})}.\\
\end{array}
\end{equation*} Let us choose a maximal chain $\mathcal{C}$ of $[v,w]$ containing $u$. That is,
$$\mathcal{C}\colon v=x_{(0)}\lessdot x_{(1)}\lessdot \cdots \lessdot x_{(k)}\lessdot \cdots \lessdot x_{(\ell)}=w \text{ and }x_{(k)}=u.$$ If $\mathcal{C}_{-}: x_{(0)}\lessdot x_{(1)}\lessdot \cdots \lessdot x_{(k)}$ and $\mathcal{C}_{+}: \lessdot x_{(k)}\lessdot \cdots \lessdot x_{(\ell)}$ are the subchains of $\mathcal{C}$, then $B_{v,u}=B(G^{\mathcal{C_{-}}})$ and $B_{u,w}=B(G^{\mathcal{C_{+}}})$. Since $B(G^{\mathcal{C}})=B(G^{\mathcal{C_{-}}}) \ast B(G^{\mathcal{C_{+}}})$, we  conclude that $B_{v,u}\ast B_{u,w}=B_{v,w}$.
\end{proof}

\section{Toric Bruhat interval polytopes}\label{sec:Toric Bruhat interval polytopes}
Recall that a Bruhat interval polytope $Q_{v,w}$ is toric if $\dim Q_{v,w}=\ell(w)-\ell(v)$. In this section, we show that the combinatorial type of a toric Bruhat interval polytope $Q_{v,w}$ is  determined by the poset structure of the interval $[v,w]$ (see Theorem~\ref{prop:3-2}). Furthermore, a toric Bruhat interval polytope is simple if and only if it is combinatorially equivalent to a cube (see Corollary~\ref{coro:3-8}).

We already have seen in Theorem~\ref{def:face-graph} that every face of a Bruhat interval polytope $Q_{v,w}$ can be realized by a subinterval of $[v,w]$.
We can show that the converse is also true when $Q_{v,w}$ is toric.

\begin{theorem} \label{prop:3-2}
For a Bruhat interval polytope $Q_{v,w}$,
the following are equivalent:
\begin{enumerate}
\item $Q_{v,w}$ is toric {\rm (}i.e., $\dim Q_{v,w}=\ell(w)-\ell(v)${\rm )}.
\item $Q_{x,y}$ is a face of $Q_{v,w}$ for any $[x,y]\subset [v,w]$.
\end{enumerate}
\end{theorem}

\begin{proof}
Suppose that (1) holds. 
Since $Q_{v,w}$ is toric, the Richardson variety $X_{w^{-1}}^{v^{-1}}$ is a toric variety by definition.  We note that $[x^{-1},y^{-1}]\subset [v^{-1},w^{-1}]$ if (and only if) $[x,y]\subset [v,w]$.  Since $X_{y^{-1}}^{x^{-1}}$  is a toric subvariety of $X_{w^{-1}}^{v^{-1}}$ and the moment map gives a one-to-one correspondence between toric subvarieties of $X_{w^{-1}}^{v^{-1}}$ and faces of $Q_{v,w}$, we have that $\mu(X_{y^{-1}}^{x^{-1}})=Q_{x,y}$ is a face of $\mu(X_{w^{-1}}^{v^{-1}})=Q_{v,w}$, proving (2).

Conversely, suppose (2) holds. We shall prove (1) by induction on the value of $\ell(w)-\ell(v)$. When $\ell(w)-\ell(v)=1$, (1) obviously holds. We note that for any $[p,q]\subset [x,y]$, $Q_{p,q}$ is a face of $Q_{x,y}$ by (2). If $[x,y]$ is a proper subset of $[v,w]$, then $\ell(y)-\ell(x)<\ell(w)-\ell(v)$ and hence $\dim Q_{x,y}=\ell(y)-\ell(x)$ by induction assumption.

Now we take $x=v$ and $y\lessdot w$. Then
\begin{equation} \label{eq:3-1}
\dim Q_{v,y}= \ell(w)-\ell(y)=\ell(w)-\ell(v)-1.
\end{equation}
Moreover, since $Q_{v,y}$ is a face of $Q_{v,w}$ by (2) and does not contain the vertex $w$, we have
\begin{equation} \label{eq:3-2}
\dim Q_{v,w}>\dim Q_{v,y}.
\end{equation}
It follows from \eqref{eq:3-1} and \eqref{eq:3-2} that
$\dim Q_{v,w}\ge \ell(w)-\ell(v)$.
Since the converse inequality holds by~\eqref{eq_dim_Qvw_and_length}, this proves (1).
\end{proof}

{Theorem}~\ref{prop:3-2} shows that if $Q_{v,w}$ is toric, then the poset structure of $[v,w]$ determines the face poset of $Q_{v,w}$. Hence the following corollary directly follows.

\begin{corollary} \label{coro:3-3}
If $Q_{v,w}$ is toric, then its combinatorial type is determined by the poset structure of~$[v,w]$. \end{corollary}

The assumption \lq\lq toric\rq\rq\ in the corollary above cannot be removed. Indeed, the interval $[1324,4231]$ is a Boolean algebra of rank 4 but the corresponding Bruhat interval polytope is of dimension~$3$. See Figure~\ref{fig:1324-4231}. On the other hand, for each positive integer $m$ there is an example of an interval $[v,w]$ of rank $m$ such that $[v,w]$ is a Boolean algebra and $Q_{v,w}$ is toric (so that $\rank [v,w]=\dim Q_{v,w}$), e.g. $w=s_ms_{m-1}\cdots s_2 s_1$ and $v=e$. In fact, $Q_{v,w}$ of this example is combinatorially equivalent to an $m$-cube.

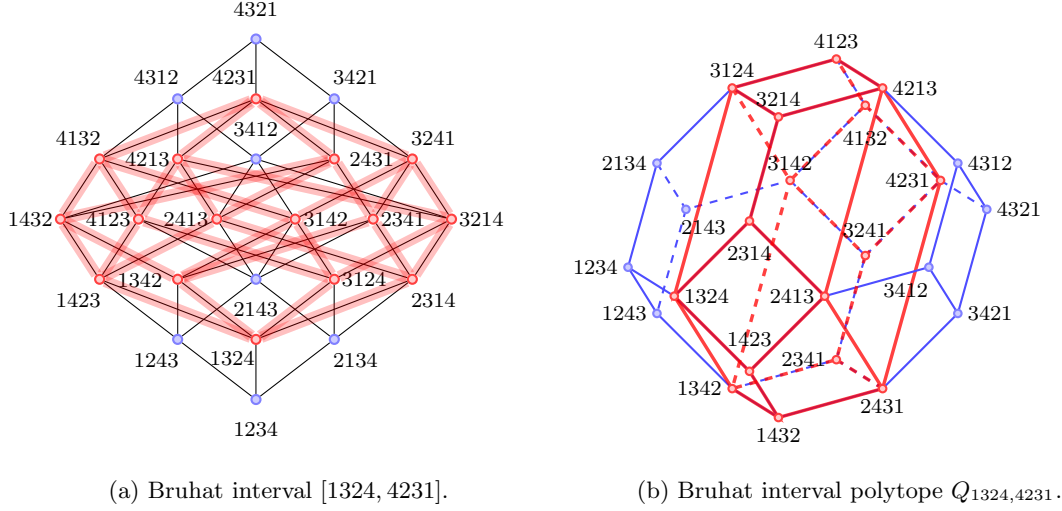
\begin{figure}
	\begin{subfigure}[b]{0.49\textwidth}
		\begin{tikzpicture}
		\tikzset{every node/.style={font=\footnotesize}}
		\tikzset{red node/.style = {fill=red!20!white, draw=red!75!white}}
		\tikzset{red line/.style = {line width=1ex, red,nearly transparent}}
		\matrix [matrix of math nodes,column sep={0.52cm,between origins},
		row sep={0.8cm,between origins},
		nodes={circle, draw=blue!50,fill=blue!20, thick, inner sep = 0pt , minimum size=1.2mm}]
		{
			& & & & & \node[label = {above:{4321}}] (4321) {} ; & & & & & \\
			& & &
			\node[label = {above left:4312}] (4312) {} ; & &
			\node[label = {above left:4231}, red node] (4231) {} ; & &
			\node[label = {above right:3421}] (3421) {} ; & & & \\
			& \node[label = {above left:4132}, red node] (4132) {} ; & &
			\node[label = {left:4213}, red node] (4213) {} ; & &
			\node[label = {above:3412}] (3412) {} ; & &
			\node[label = {[label distance = 0.1cm]0:2431}, red node] (2431) {} ; & &
			\node[label = {above right:3241}, red node] (3241) {} ; & \\
			\node[label = {left:1432}, red node] (1432) {} ; & &
			\node[label = {left:4123}, red node] (4123) {} ; & &
			\node[label = {[label distance = 0.01cm]180:2413}, red node] (2413) {} ; & &
			\node[label = {[label distance = 0.01cm]0:3142}, red node] (3142) {} ; & &
			\node[label = {right:2341}, red node] (2341) {} ; & &
			\node[label = {right:3214}, red node] (3214) {} ; \\
			& \node[label = {below left:1423}, red node] (1423) {} ; & &
			\node[label = {[label distance = 0.1cm]182:1342}, red node] (1342) {} ; & &
			\node[label = {below:2143}] (2143) {} ; & &
			\node[label = {right:3124}, red node] (3124) {} ; & &
			\node[label = {below right:2314}, red node] (2314) {} ; & \\
			& & & \node[label = {below left:1243}] (1243) {} ; & &
			\node[label = {[label distance = 0.01cm]190:1324}, red node] (1324) {} ; & &
			\node[label = {below right:2134}] (2134) {} ; & & & \\
			& & & & & \node[label = {below:1234}] (1234) {} ; & & & & & \\
		};
		
		\draw (4321)--(4312)--(4132)--(1432)--(1423)--(1243)--(1234)--(2134)--(2314)--(2341)--(3241)--(3421)--(4321);
		\draw (4321)--(4231)--(4132);
		\draw (4231)--(3241);
		\draw (4231)--(2431);
		\draw (4231)--(4213);
		\draw (4312)--(4213)--(2413)--(2143)--(3142)--(3241);
		\draw (4312)--(3412)--(2413)--(1423)--(1324)--(1234);
		\draw (3421)--(3412)--(3214)--(3124)--(1324);
		\draw (3421)--(2431)--(2341)--(2143)--(2134);
		\draw (4132)--(4123)--(1423);
		\draw (4132)--(3142)--(3124)--(2134);
		\draw (4213)--(4123)--(2143)--(1243);
		\draw (4213)--(3214);
		\draw (3412)--(1432)--(1342)--(1243);
		\draw (2431)--(1432);
		\draw (2431)--(2413)--(2314);
		\draw (3142)--(1342)--(1324);
		\draw (4123)--(3124);
		\draw (2341)--(1342);
		\draw (2314)--(1324);
		\draw (3412)--(3142);
		\draw (3241)--(3214)--(2314);
		
		\draw[red line] (1324)--(1423)--(1432)--(4132)--(4231);
		\draw[red line] (1324)--(2314)--(3214)--(3241)--(4231);
		\draw[red line] (1324)--(1342)--(1432);
		\draw[red line] (1324)--(3124)--(3214);
		\draw[red line] (2314)--(2341)--(2431)--(4231);
		\draw[red line] (1423)--(4123)--(4213)--(4231);
		\draw[red line](4123)--(4132);
		\draw[red line] (3214)--(4213);
		\draw[red line] (2341)--(3241);
		\draw[red line](1432)--(2431);
		\draw[red line] (1342)--(2341);
		\draw[red line] (3124)--(4123);
		\draw[red line] (1423)--(2413)--(4213);
		\draw[red line] (3124)--(3142)--(4132);
		\draw[red line] (3142)--(3241);
		\draw[red line] (1342)--(3142);
		\draw[red line] (2314)--(2413);
		\draw[red line] (2413)--(2431);
		
		\end{tikzpicture}
		\caption{Bruhat interval $[1324, 4231]$.}
	\end{subfigure}~
	\begin{subfigure}[b]{0.49\textwidth}
	\begin{tikzpicture}[scale=6]
		\tikzset{every node/.style={draw=blue!50,fill=blue!20, circle, thick, inner sep=1pt,font=\footnotesize}}
\tikzset{red node/.style = {fill=red!20!white, draw=red!75!white}}
\tikzset{red line/.style = {line width=0.3ex, red, nearly opaque}}
	
	\coordinate (3142) at (1/3, 1/2, 1/6);
	\coordinate (4231) at (2/3, 1/2, 1/6);
	\coordinate (4312) at (5/6, 2/3, 1/2);
	\coordinate (4321) at (5/6, 1/2, 1/3);
	\coordinate (3421) at (5/6, 1/3, 1/2);
	\coordinate (4213) at (2/3, 5/6, 1/2);
	\coordinate (1324) at (1/3, 1/2, 5/6);
	\coordinate (2413) at (2/3, 1/2, 5/6);
	\coordinate (3412) at (5/6, 1/2, 2/3);
	\coordinate (2314) at (1/2, 2/3, 5/6);
	\coordinate (4123) at (1/2, 5/6, 1/3);
	\coordinate (4132) at (1/2, 2/3, 1/6);
	\coordinate (3214) at (1/2, 5/6, 2/3);
	\coordinate (3124) at (1/3, 5/6, 1/2);
	\coordinate (2431) at (2/3, 1/6, 1/2);
	\coordinate (1432) at (1/2, 1/6, 2/3);
	\coordinate (1423)  at (1/2, 1/3, 5/6);
	\coordinate (1342)  at (1/3, 1/6, 1/2);
	\coordinate (2341) at (1/2, 1/6, 1/3);
	\coordinate (3241) at (1/2, 1/3, 1/6);
	\coordinate (1243) at (1/6, 1/3, 1/2);
	\coordinate (2143) at (1/6, 1/2, 1/3);
	\coordinate (1234) at (1/6, 1/2, 2/3);
	\coordinate (2134) at (1/6, 2/3, 1/2);
	
	\draw[red line] (2431)--(4231)--(4213);
	
	\draw[thick, draw=blue!70] (4213)--(4312)--(3412)--(2413)--(2314)--(3214)--cycle;
	\draw[thick, draw=blue!70] (4312)--(4321)--(3421)--(3412);
	\draw[thick, draw=blue!70] (3421)--(2431)--(1432)--(1423)--(2413);
	\draw[thick, draw=blue!70] (1423)--(1324)--(2314);
	\draw[thick, draw=blue!70] (1432)--(1342)--(1243)--(1234)--(1324);
	\draw[thick, draw=blue!70] (1234)--(2134)--(3124)--(3214);
	\draw[thick, draw=blue!70] (3124)--(4123)--(4213);
	
	\draw[thick, draw=blue!70, dashed] (2134)--(2143)--(3142)--(4132)--(4123);
	\draw[thick, draw=blue!70, dashed] (2143)--(1243);
	\draw[thick, draw=blue!70, dashed] (3142)--(3241)--(2341)--(1342);
	\draw[thick, draw=blue!70, dashed] (2341)--(2431);
	\draw[thick, draw=blue!70, dashed] (3241)--(4231)--(4132);
	\draw[thick, draw=blue!70, dashed] (4231)--(4321);

	\draw[red line] (2314)--(2413)--(4213)--(3214)--cycle;
	\draw[red line] (3214)--(3124)--(4123)--(4213);
	\draw[red line] (3124)--(1324)--(2314);
	\draw[red line] (1324)--(1423)--(2413);
	\draw[red line] (1324)--(1342)--(1432)--(1423);
	\draw[red line] (2413)--(2431)--(1432);

	\draw[red line, dashed] (3124)--(3142)--(4132)--(4123);
	\draw[red line, dashed] (4132)--(4231)--(3241)--(3142);
	\draw[red line, dashed] (3241)--(2341)--(1342)--(3142);
	\draw[red line, dashed] (2341)--(2431);

\node [label = {[label distance = 0cm]left:1234}] at (1234) {};
\node[label = {[label distance = 0cm]left:1243}] at (1243) {};
\node[label = {[label distance = 0cm]right:1324}, red node] at (1324) {};
\node[label = {[label distance = 0cm]left:1342}, red node] at (1342) {};
\node [label = {[label distance = 0cm]above:1423}, red node] at (1423) {};
\node[label = {[label distance = -0.2cm]below:1432}, red node] at (1432) {};
\node [label = {[label distance = 0cm]left:2134}] at (2134) {};
\node[label = {[label distance = -0.1cm]below right:2143}] at (2143) {};
\node[label = {[label distance = 0cm]below:2314}, red node] at (2314) {};
\node[label = {[label distance = 0cm]left:2341}, red node] at (2341) {};
\node[label = {[label distance = 0cm]left:2413}, red node] at (2413) {};
\node[label = {[label distance = -0.2cm]below:2431}, red node] at (2431) {};
\node[label = {[label distance = -0.2cm]above:3124}, red node] at (3124) {};
\node[label = {[label distance = -0.2cm]above:3142}, red node] at (3142) {};
\node[label = {[label distance = -0.2cm]above:3214}, red node] at (3214) {};
\node [label = {[label distance = -0.1cm]above:3241}, red node] at (3241) {};
\node[label = {[label distance = 0cm]below left:3412}] at (3412) {};
\node[label = {[label distance = 0cm]right:3421}] at (3421) {};
\node[label = {[label distance = -0.2cm]above:4123}, red node] at (4123) {};
\node [label = {[label distance = 0cm]below:4132}, red node] at (4132) {};
\node[label = {[label distance = 0cm]right:4213}, red node] at (4213) {};
\node[label = {[label distance = 0cm]left:4231}, red node] at (4231) {};
\node[label = {[label distance = 0cm]right:4312}] at (4312) {};
\node [label = {[label distance = 0cm]right:4321}] at (4321) {};

	\end{tikzpicture}
	\caption{Bruhat interval polytope $Q_{1324,4231}$.}
\end{subfigure}
	\caption{The interval $[1324,4231]$ is a Boolean algebra but $Q_{1324,4231}$ is not a cube.}\label{fig:1324-4231}
\end{figure}

\begin{corollary}\label{coro:combi-equiv}
If $Q_{v,w}$ is toric, then $Q_{v,w}$ and $Q_{v^{-1},w^{-1}}$ are combinatorially equivalent.
\end{corollary}

\begin{remark}
For $n\leq 4$ and for every pair of~$v$ and~$w$ in $\mathfrak{S}_n$ with $v\leq w$, $Q_{v,w}$ is combinatorially equivalent to $Q_{v^{-1},w^{-1}}$. But for $n=5$ there are 160 pairs of~$v$ and~$w$ such that $Q_{v,w}$ is not combinatorially equivalent to $Q_{v^{-1},w^{-1}}$.
\end{remark}

\noindent
{\bf Convention.}
In the following, when a polytope~$Q$ is combinatorially equivalent to a cube (or a $d$-cube), we simply say that~$Q$ is a cube (or a $d$-cube). We also say that an interval $[v,w]$ is \emph{Boolean} if it is a Boolean algebra.

\medskip

We recall the following fact from \cite[Problems and Exercises 0.1 in p.23]{zieg98} (see also \cite[Appendix]{yu-ma16}):

\begin{lemma} \label{lemm:3-4}
If $Q$ is a simple polytope of dimension $\ge 2$ and every $2$-face of $Q$ is a $2$-cube, then $Q$ is a cube.
\end{lemma}

This lemma implies the following.

\begin{proposition} \label{prop:3-5}
Suppose that $Q_{v,w}$ is toric. Then $Q_{v,w}$ is a cube if and only if it is simple. {\rm(}This is equivalent to saying that a smooth toric Richardson variety is a Bott manifold{.\rm)}
\end{proposition}

\begin{proof}
Since the \lq\lq only if\rq\rq\ part is trivial, it suffices to prove the \lq\lq if \rq\rq part. Since $Q_{v,w}$ is toric, every $k$-face of $Q_{v,w}$ corresponds to a $k$-interval in $[v,w]$. Since every $2$-interval is a diamond (see~\cite[Lemma 2.7.3]{BB_combinatorics05}), every $2$-face of $Q_{v,w}$ is a $2$-cube. Therefore, if $Q_{v,w}$ is simple, then $Q_{v,w}$ is a cube by Lemma~\ref{lemm:3-4}.
\end{proof}

\begin{figure}[t]
	\begin{subfigure}[b]{0.49\textwidth}
		\begin{tikzpicture}[scale=.7]
		\tikzset{every node/.style={font=\footnotesize}}
		\matrix [matrix of math nodes,column sep={0.52cm,between origins},
		row sep={0.8cm,between origins},
		nodes={circle, draw=blue!50,fill=blue!20, thick, inner sep = 0pt , minimum size=1.2mm}]
		{
			& & & & & \node[label = {above:{4321}}] (4321) {} ; & & & & & \\
			& & &
			\node[label = {above left:4312}] (4312) {} ; & &
			\node[label = {above left:4231}] (4231) {} ; & &
			\node[label = {above right:3421}] (3421) {} ; & & & \\
			& \node[label = {above left:4132},fill=red!20!white, draw=red!75!white] (4132) {} ; & &
			\node[label = {left:4213}] (4213) {} ; & &
			\node[label = {above:3412}] (3412) {} ; & &
			\node[label = {[label distance = 0.1cm]0:2431}] (2431) {} ; & &
			\node[label = {above right:3241}] (3241) {} ; & \\
			\node[label = {left:1432}, fill=red!20!white, draw=red!75!white] (1432) {} ; & &
			\node[label = {left:4123}, fill=red!20!white, draw=red!75!white] (4123) {} ; & &
			\node[label = {[label distance = 0.01cm]180:2413}] (2413) {} ; & &
			\node[label = {[label distance = 0.01cm]0:3142}, fill=red!20!white, draw=red!75!white] (3142) {} ; & &
			\node[label = {right:2341}] (2341) {} ; & &
			\node[label = {right:3214}] (3214) {} ; \\
			& \node[label = {below left:1423}, fill=red!20!white, draw=red!75!white] (1423) {} ; & &
			\node[label = {[label distance = 0.1cm]182:1342}, fill=red!20!white, draw=red!75!white] (1342) {} ; & &
			\node[label = {below:2143}, fill=red!20!white, draw=red!75!white] (2143) {} ; & &
			\node[label = {right:3124}] (3124) {} ; & &
			\node[label = {below right:2314}] (2314) {} ; & \\
			& & & \node[label = {below left:1243}, fill=red!20!white, draw=red!75!white ] (1243) {} ; & &
			\node[label = {[label distance = 0.01cm]190:1324}] (1324) {} ; & &
			\node[label = {below right:2134}] (2134) {} ; & & & \\
			& & & & & \node[label = {below:1234}] (1234) {} ; & & & & & \\
		};
		
		\draw (4321)--(4312)--(4132)--(1432)--(1423)--(1243)--(1234)--(2134)--(2314)--(2341)--(3241)--(3421)--(4321);
		\draw (4321)--(4231)--(4132);
		\draw (4231)--(3241);
		\draw (4231)--(2431);
		\draw (4231)--(4213);
		\draw (4312)--(4213)--(2413)--(2143)--(3142)--(3241);
		\draw (4312)--(3412)--(2413)--(1423)--(1324)--(1234);
		\draw (3421)--(3412)--(3214)--(3124)--(1324);
		\draw (3421)--(2431)--(2341)--(2143)--(2134);
		\draw (4132)--(4123)--(1423);
		\draw (4132)--(3142)--(3124)--(2134);
		\draw (4213)--(4123)--(2143)--(1243);
		\draw (4213)--(3214);
		\draw (3412)--(1432)--(1342)--(1243);
		\draw (2431)--(1432);
		\draw (2431)--(2413)--(2314);
		\draw (3142)--(1342)--(1324);
		\draw (4123)--(3124);
		\draw (2341)--(1342);
		\draw (2314)--(1324);
		\draw (3412)--(3142);
		\draw (3241)--(3214)--(2314);

		\draw[line width=1ex, red,nearly transparent] (1243)--(1423)--(1432)--(4132)--(4123)--(1423);
		\draw[line width=1ex, red,nearly transparent] (1243)--(1342)--(1432);
		\draw[line width=1ex, red,nearly transparent] (1243)--(2143)--(3142)--(4132);
		\draw[line width=1ex, red,nearly transparent] (2143)--(4123);
		\draw[line width=1ex, red,nearly transparent] (1342)--(3142);
		
		\end{tikzpicture}
		\subcaption{Bruhat interval $[1243, 4132]$.}
	\end{subfigure}
	\begin{subfigure}[b]{0.49\textwidth}
		\begin{tikzpicture}[scale=6]
		\tikzset{every node/.style={draw=blue!50,fill=blue!20, circle, thick, inner sep=1pt,font=\footnotesize}}
		\tikzset{red node/.style = {fill=red!20!white, draw=red!75!white}}
		\tikzset{red line/.style = {line width=0.3ex, red, nearly opaque}}
		
		\coordinate (3142) at (1/3, 1/2, 1/6);
		\coordinate (4231) at (2/3, 1/2, 1/6);
		\coordinate (4312) at (5/6, 2/3, 1/2);
		\coordinate (4321) at (5/6, 1/2, 1/3);
		\coordinate (3421) at (5/6, 1/3, 1/2);
		\coordinate (4213) at (2/3, 5/6, 1/2);
		\coordinate (1324) at (1/3, 1/2, 5/6);
		\coordinate (2413) at (2/3, 1/2, 5/6);
		\coordinate (3412) at (5/6, 1/2, 2/3);
		\coordinate (2314) at (1/2, 2/3, 5/6);
		\coordinate (4123) at (1/2, 5/6, 1/3);
		\coordinate (4132) at (1/2, 2/3, 1/6);
		\coordinate (3214) at (1/2, 5/6, 2/3);
		\coordinate (3124) at (1/3, 5/6, 1/2);
		\coordinate (2431) at (2/3, 1/6, 1/2);
		\coordinate (1432) at (1/2, 1/6, 2/3);
		\coordinate (1423)  at (1/2, 1/3, 5/6);
		\coordinate (1342)  at (1/3, 1/6, 1/2);
		\coordinate (2341) at (1/2, 1/6, 1/3);
		\coordinate (3241) at (1/2, 1/3, 1/6);
		\coordinate (1243) at (1/6, 1/3, 1/2);
		\coordinate (2143) at (1/6, 1/2, 1/3);
		\coordinate (1234) at (1/6, 1/2, 2/3);
		\coordinate (2134) at (1/6, 2/3, 1/2);
		
		\draw[thick, draw=blue!70, dashed] (2134)--(2143)--(3142)--(4132)--(4123);
		\draw[thick, draw=blue!70, dashed] (2143)--(1243);
		\draw[thick, draw=blue!70, dashed] (3142)--(3241)--(2341)--(1342);
		\draw[thick, draw=blue!70, dashed] (2341)--(2431);
		\draw[thick, draw=blue!70, dashed] (3241)--(4231)--(4132);
		\draw[thick, draw=blue!70, dashed] (4231)--(4321);
		
		\draw[red line] (1243)--(2143)--(2413);
		
		\draw[thick, draw=blue!70] (4213)--(4312)--(3412)--(2413)--(2314)--(3214)--cycle;
		\draw[thick, draw=blue!70] (4312)--(4321)--(3421)--(3412);
		\draw[thick, draw=blue!70] (3421)--(2431)--(1432)--(1423)--(2413);
		\draw[thick, draw=blue!70] (1423)--(1324)--(2314);
		\draw[thick, draw=blue!70] (1432)--(1342)--(1243)--(1234)--(1324);
		\draw[thick, draw=blue!70] (1234)--(2134)--(3124)--(3214);
		\draw[thick, draw=blue!70] (3124)--(4123)--(4213);
		
		\draw[red line, dashed] (2143)--(2341)--(2431);
		\draw[red line, dashed] (2341)--(1342);
		
		\draw[red line] (1423)--(1432)--(2431)--(2413)--cycle;
		\draw[red line] (1423)--(1243);
		\draw[red line] (1243)--(1342)--(1432);
		
		\node [label = {[label distance = 0cm]left:1234}] at (1234) {};
		\node[label = {[label distance = 0cm]left:1243}, red node] at (1243) {};
		\node[label = {[label distance = 0cm]right:1324}] at (1324) {};
		\node[label = {[label distance = 0cm]left:1342}, red node] at (1342) {};
		\node [label = {[label distance = 0cm]above:1423},red node] at (1423) {};
		\node[label = {[label distance = -0.2cm]below:1432}, red node] at (1432) {};
		\node [label = {[label distance = 0cm]left:2134}] at (2134) {};
		\node[label = {[label distance = -0.1cm]above right:2143}, red node] at (2143) {};
		\node[label = {[label distance = 0cm]right:2314}] at (2314) {};
		\node[label = {[label distance = -0.1cm]below left:2341}, red node] at (2341) {};
		\node[label = {[label distance = -0.1cm]above :2413}, red node] at (2413) {};
		\node[label = {[label distance = -0.2cm]below:2431}, red node] at (2431) {};
		\node[label = {[label distance = -0.2cm]above:3124}] at (3124) {};
		\node[label = {[label distance = -0.2cm]above:3142}] at (3142) {};
		\node[label = {[label distance = -0.2cm]above:3214}] at (3214) {};
		\node [label = {[label distance = -0.1cm]above:3241}] at (3241) {};
		\node[label = {[label distance = 0cm]below left:3412}] at (3412) {};
		\node[label = {[label distance = 0cm]right:3421}] at (3421) {};
		\node[label = {[label distance = -0.2cm]above:4123}] at (4123) {};
		\node [label = {[label distance = 0cm]below:4132}] at (4132) {};
		\node[label = {[label distance = 0cm]right:4213}] at (4213) {};
		\node[label = {[label distance = 0cm]left:4231}] at (4231) {};
		\node[label = {[label distance = 0cm]right:4312}] at (4312) {};
		\node [label = {[label distance = 0cm]right:4321}] at (4321) {};

		\end{tikzpicture}
		\subcaption{ Bruhat interval polytope $Q_{1243, 2431}$.}
	\end{subfigure}
	\caption{An example of a Bruhat interval polytope which is a cube.}\label{fig:1243-2431}
\end{figure}
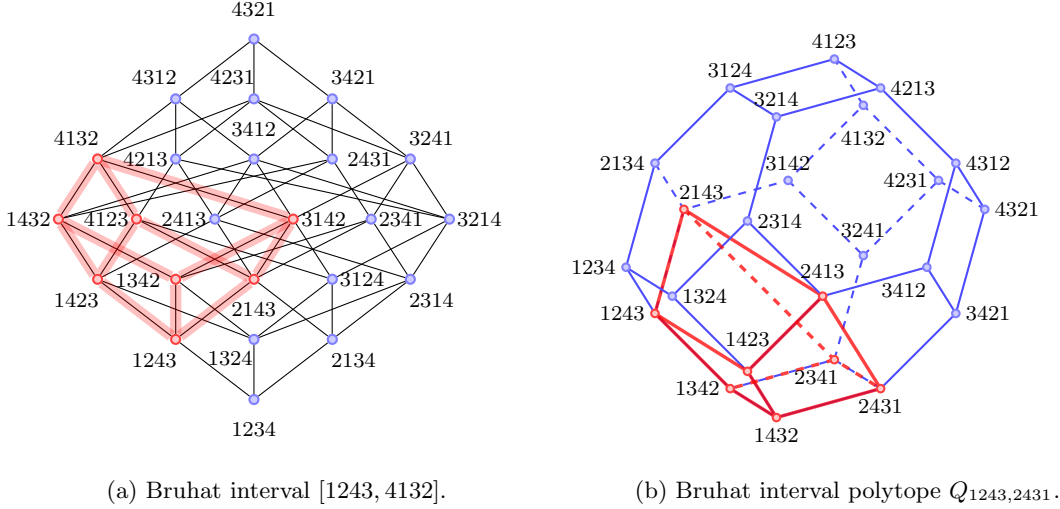

The following gives a characterization of when $Q_{v,w}$ is a cube and also proves Theorem~\ref{thm:main}.

\begin{theorem} \label{theo:3-6}
A Bruhat interval polytope $Q_{v,w}$ is a cube if and only if it is toric and $[v,w]$ is Boolean. {\rm(}This is equivalent to saying that a Richardson variety $X^v_w$ is a Bott manifold if and only if it is toric and $[v,w]$ is Boolean{.\rm)}
\end{theorem}
In Figure~\ref{fig:1243-2431}, one can see an example of a Bruhat interval polytope $Q_{v,w}$ which is  toric and the interval $[v,w]$ is Boolean.
Combining the above theorem with Corollary~\ref{coro:combi-equiv}, we get the following.
\begin{corollary}\label{coro:cube-equiv}
A Bruhat interval polytope $Q_{v,w}$ is a cube if and only if $Q_{v^{-1},w^{-1}}$ is a cube.
\end{corollary}

\begin{proof}[Proof of Theorem~\ref{theo:3-6}]
First we prove the \lq\lq if\rq\rq\ part. Since $[v,w]$ is Boolean, every element $u\in [v,w]$ covers $\ell(u)-\ell(v)$ elements and is covered by $\ell(w)-\ell(u)$ elements. Since $Q_{v,w}$ is toric, $\dim Q_{v,w}=\ell(w)-\ell(v)$ and any edge incident to the vertex $u$ in $Q_{v,w}$ is obtained from the cover relations by {Theorem}~\ref{prop:3-2}. This shows that $Q_{v,w}$ is simple. Therefore, $Q_{v,w}$ is a cube by Proposition~\ref{prop:3-5}.

The rest of the proof is devoted to the proof of the \lq\lq only if\rq\rq\ part. We shall prove it by induction on the dimension $m$ of the cube. It is obvious when $m=1$. Suppose that it holds for $m-1$ and that $Q_{v,w}$ is an $m$-cube. Then $Q_{v,w}$ has two disjoint facets, both of which are an $(m-1)$-cube. We denote those facets by $Q_{p,q}$ and $Q_{r,s}$. Then
\begin{enumerate}
\item $[p,q]\cap [r,s]=\emptyset$, $[p,q]\cup [r,s]=[v,w]$, and
\item $[p,q]$ and $[r,s]$ are both Boolean by induction assumption.
\end{enumerate}
By (1), we may assume that $p=v$ and $s=w$ without loss of generality. Then
\begin{equation} \label{eq:3-3}
\ell(q)=\ell(w)-1\quad\text{and}\quad \ell(r)=\ell(v)+1.
\end{equation}
Indeed, if $\ell(q)\le \ell(w)-2$, then no element in $[v,q]$ is covered by $w$ while there are exactly $m-1$ elements in $[r,w]$ covered by $w$ because $[r,w]$ is Boolean and of rank $m-1$. Therefore, the number of coatoms in $[v,w]$ is $m-1$. On the other hand, since $Q_{v,w}$ is an $m$-cube, there must be $m$ coatoms in $[v,w]$. This is a contradiction. Therefore, $\ell(q)=\ell(w)-1$. A similar argument shows that $\ell(r)=\ell(v)+1$.

Since $Q_{v,q}$ is an $(m-1)$-cube, we have $\dim Q_{v,q}=\ell(q)-\ell(v)$ by induction assumption. Here $\ell(q)=\ell(w)-1$ by \eqref{eq:3-3} and $\dim Q_{v,q}+1=\dim Q_{v,w}$ since $Q_{v,q}$ is a facet of $Q_{v,w}$. These show that $m=\dim Q_{v,w}=\ell(w)-\ell(v)$, i.e., $Q_{v,w}$ is toric.

We need to show that $[v,w]$ is Boolean. Since $Q_{v,w}$ is toric, any interval $[x,y]\subset [v,w]$ produces a face of $Q_{v,w}$ by {Theorem}~\ref{prop:3-2}.
This in particular shows that the numbers of atoms and coatoms in $[v,w]$ are both $m=\ell(w)-\ell(v)$ since $Q_{v,w}$ is an $m$-cube. Moreover, since $Q_{v,w}$ is a cube, so is $Q_{x,y}$ and its dimension is $\ell(y)-\ell(x)$ as observed above.
Therefore if $[x,y]$ is a proper subset of $[v,w]$, then one can apply the induction assumption to $Q_{x,y}$ so that $[x,y]$ is Boolean.
In particular, for $u\in (v,w)$, $[v,u]$ and $[u,w]$ are proper subsets of $[v,w]$ so that they are Boolean. This means that $u$ covers exactly $\ell(u)-\ell(v)$ elements and is covered by exactly $\ell(w)-\ell(u)$ elements. (This holds even when $u=v$ or $u=w$ as observed above.)
This almost shows that $[v,w]$ is Boolean but we have to observe cover relations among elements to conclude that it is Boolean.

For $u\in [v,w]$ we set $\rank u=\ell(u)-\ell(v)$.
As observed at the beginning of the proof of the \lq\lq only if\rq\rq\ part, there are elements $q, r\in [v,w]$ such that
\begin{enumerate}
\item $\rank q=m-1$, $\rank r=1$, and
\item $[v,q]\cap [r,w]=\emptyset$, $[v,q]\cup [r,w]=[v,w]$.
\end{enumerate}
Since $[v,q]$ is Boolean and of rank $m-1$ and $\rank v=0$, we may regard $[v,q]$ as the Boolean algebra obtained from the set $[m-1]$. Then the atoms of $[v,w]$ are $\{1\},\{2\},\dots,\{m\}$. Since $\rank r=1$, we regard $r$ as $\{m\}$.  Since $[r,w]$ is Boolean and of rank $m-1$ and $\rank r=1$, there are $m-1$ elements of rank $2$ in $[r,w]$ and each of them covers $r(=\{m\})$ but since they are of rank $2$, each of them must cover one more element which is in $[v,q]$. We denote by $\{i,m\}$ $(i\in [m-1])$ the rank 2 element in $[r,w]$ which covers the atoms $\{i\}$ and $\{m\}$.
Since $[r,w]$ is Boolean and of rank $m-1$, we may regard rank $k+1$ elements in $[r,w]$ as $\{i_1,\dots,i_{k},m\}$ where $\{i_1,\dots,i_{k}\}$ is a subset of $[m-1]$ and $\{i_1,\dots,i_{k},m\}$ covers $\{i_1,\dots,\widehat{i_j},\dots,i_k,m\}$ $(1\le j\le k)$.

The element $\{i_1,\dots,i_{k},m\}$ is of rank $k+1$ and already covers $k$ elements $\{i_1,\dots,\widehat{i_j},\dots,i_k,m\}$ $(1\le j\le k)$. Therefore, it suffices to show that
\[
(*) \qquad \text{$\{i_1,\dots,i_k,m\}$ covers the element $\{i_1,\dots,i_k\}$ in $[v,q]$.}
\]
When $k=m-1$, the element $\{i_1,\dots,i_k,m\}$ is the entire set $\{1,\dots,m\}$ (that is $w$). In this case we already know that it covers all the coatoms of $[v,w]$. Therefore, we may assume $k<m-1$.

We shall prove $(*)$ above by induction on $k$. When $k=1$, $\{i_1,m\}$ covers $\{i_1\}$ by definition. Suppose that $(*)$ holds for $k-1$ and $2\le k<m-1$.
We look at the interval $I$ between the empty set (that is $v$) and $\{i_1,\dots,i_k,m\}$. Since $k<m-1$, $I$ is a proper subset of $[v,w]$; so $I$ is Boolean and we know that $\{i_1,\dots,\widehat{i_j},\dots,i_k,m\}$ $(1\le j\le k)$ are all in $I$. Therefore, $I$ contains all subsets of $\{i_1,\dots,i_k,m\}$ except $\{i_1,\dots,i_k\}$ which follows from induction assumption and the fact that $I$ is Boolean. The only missing element in $I$ lies in $[v,q]$ and is of rank $k$, i.e., of the form $\{j_1,\dots,j_k\}$ where $\{j_1,\dots,j_k\}$ is a subset of $[m-1]$. Here, $\{j_1,\dots,j_k\}$ must contain all proper subsets of $\{i_1,\dots,i_k\}$ since those proper subsets are in $I$ and $I$ is Boolean. Therefore, $\{j_1,\dots,j_k\}=\{i_1,\dots,i_k\}$. This completes the induction step and the proof of the proposition.
\end{proof}

The following examples show that one cannot drop either \emph{toric} or \emph{Boolean} in Theorem~\ref{theo:3-6}.

\begin{example} \label{exam:3-1}
\begin{enumerate}
\item Let $v=1324$ and $w=4231$. Then $[v,w]$ is Boolean of length $4$ but since $\dim Q_{v,w}=3$, $Q_{v,w}$ is not toric. The vertices $v$ and $w$ have degree $4$, so $Q_{v,w}$ is not a cube.
\item Let $v=1324$ and $w=3412$. Then $v=s_2$, $w=vs_3s_1s_2=vs_1s_3s_2$, and $\dim Q_{v,w}=3$. Hence $Q_{v,w}$ is toric. But $[v,w]$ is a $4$-crown and not Boolean. The vertices $v$ and $w$ have degree $4$ and the others are simple vertices, so $Q_{v,w}$ is not a cube.

Similarly, if $v=2143$ and $w=4231$, then $v=s_1s_3=s_3s_1$, $w=vs_2s_3s_1=vs_2s_1s_3$, and $\dim Q_{v,w}=3$. Hence $Q_{v,w}$ is toric but $[v,w]$ is a $4$-crown and not Boolean.
\end{enumerate}
\end{example}

It was shown in~\cite[Theorem 3.5.2]{re02} that the largest rank of Boolean Bruhat intervals in $\mathfrak{S}_{n+1}$ is at least $n+\lfloor\frac{n-1}{2}\rfloor$ by finding a sufficient condition on $v$ and $w$ for $[v,w]$ to be Boolean. This implies that there are many Boolean Bruhat intervals which are not toric like  the Boolean interval $[1324,4231]$.

The following proposition geometrically means that a toric Richardson variety $X^v_w$ is smooth if it is smooth at either $vB$ or $wB$:

\begin{proposition} \label{prop:3-7}
If $Q_{v,w}$ is toric and either $v$ or $w$ is a simple vertex, then $Q_{v,w}$ is simple. \textup{(}Indeed, $Q_{v,w}$ is a cube by Proposition~\ref{prop:3-5}.\textup{)}
\end{proposition}

\begin{proof}
We set $m=\ell(w)-\ell(v)$ and prove the proposition by  induction on $m$. Suppose that the vertex $v$ is simple (the same argument works when the vertex $w$ is simple). We shall prove that $w$ is also a simple vertex.
Let $y$ be a coatom in $[v,w]$. Since $Q_{v,w}$ is toric, $Q_{v,y}$ is a facet by Theorem~\ref{prop:3-2}. Since $v$ is a simple vertex and $\dim Q_{v,w}=m$, there are exactly $m$ atoms and the direction vectors of the $m$ edges incident to $v$ are linearly independent. This means that $Q_{v,y}$ contains exactly $m-1$ edges incident to $v$ since $\dim Q_{v,y}=m-1$. Therefore, the number of coatoms in $[v,w]$ is at most $\binom{m}{m-1}=m$. However, since $\dim Q_{v,w}=m$, there must be at least $m$ coatoms. Therefore the number of coatoms in $[v,w]$ is $m$, which means that $w$ is a simple vertex.

Let $u\in (v,w)$. Since $Q_{v,w}$ is toric, both $Q_{v,u}$ and $Q_{u,w}$ are toric; so $\dim Q_{v,u}=\ell(u)-\ell(v)$ and $\dim Q_{u,w}=\ell(w)-\ell(u)$ and they are strictly less than $m$. As observed above, $v$ and $w$ are simple vertices of $Q_{v,w}$ and this means that they are also simple vertices of $Q_{v,u}$ and $Q_{u,w}$. Therefore one can apply the induction assumption to $Q_{v,u}$ and $Q_{u,w}$, so that they are both simple. These show that $u$ covers $\ell(u)-\ell(v)$ elements and is covered by $\ell(w)-\ell(u)$ elements. Hence $u$ is a simple vertex of $Q_{v,w}$, proving the proposition.
\end{proof}

	Example~\ref{exam:3-1}(2) shows that there is a toric Bruhat interval polytope $Q_{v,w}$ such that
	\[
	d(v)=d(w)=4>\dim Q_{v,w}=3.
	\]
	 Here is an example of a toric Bruhat interval polytope $Q_{v,w}$ with $d(v)\not=d(w)$.
\begin{example} \label{exam:3-2}
Let $v=13254$ and $w=35142$.  Since $v=s_2s_4$ and $w=vs_3s_4s_1s_2$, we get $\ell(w)-\ell(v)=4$. It follows from~\eqref{eq_dim_Qvw_Bvw} that
\[
\dim Q_{v,w}=5-\#B_{v,w}=4.
\]
Hence the polytope  $Q_{v,w}$ is toric. As we can see in Figure~\ref{fig:13254-35142}, there are six atoms and five coatoms. By Theorem~\ref{prop:3-2},  $d(v)=6$ and $d(w)=5$.
We note that $Q_{v,w}$ is not a cube.

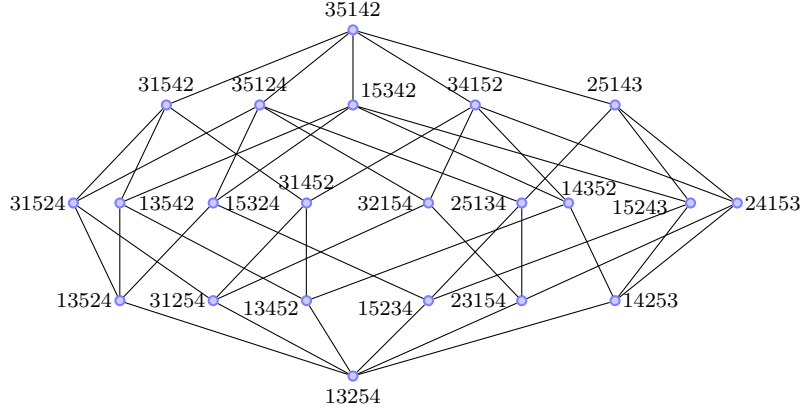
\begin{figure}[h!]
		\begin{tikzpicture}
\tikzset{every node/.style={font=\footnotesize}}
		\matrix [matrix of math nodes,column sep={0.62cm,between origins},
		row sep={1cm,between origins},
		nodes={circle, draw=blue!50,fill=blue!20, thick, inner sep = 0pt , minimum size=1.2mm}]
		{			
			&&&&&& \node[label = {[yshift=-0.2cm]above:{35142}}] (35142) {}; \\
			&& \node[label = {[yshift=-0.2cm]above:{31542}}] (31542) {} ;&&
				\node[label = {[yshift=-0.2cm]above:{35124}}] (35124) {}; &&
				\node[label = {[yshift=0.2cm]right:{15342}}] (15342) {} ;&&
				\node[label = {[yshift=-0.2cm]above:{34152}}] (34152) {}; &&&
				\node[label ={[yshift=-0.2cm]above:{25143}}] (25143) {}; \\	[2ex]
			\node[label = {left:{31524}}] (31524) {};
				& \node[label = {[xshift=0.15cm]right:{13542}}] (13542) {};
				&& \node[label = {[xshift = 0.05cm]right:{15324}}] (15324) {};
				&& \node[label = {[yshift=-0.2cm]above:{31452}}] (31452) {};
				&& \node[label = {[xshift = -0.1cm]left:{32154}}] (32154) {};
				&& \node[label = {[xshift = -0.1cm]left:{25134}}] (25134) {};
				& \node[label = {[yshift=0.2cm,xshift=-0.2cm]right:{14352}}] (14352) {};
				&& \node[label = {[yshift=-0.05cm, xshift=-0.2cm]left:{15243}}] (15243) {};
				& \node[label = {right:{24153}}] (24153) {}; \\ [2ex]
			& \node[label = {left:{13524}}] (13524) {};
				&&\node[label = {left:{31254}}] (31254) {};
				&& \node[label = {[yshift=-0.1cm]left:{13452}}] (13452) {};
				&& \node[label = {[yshift=-0.1cm, xshift=-0.1cm]left:{15234}}] (15234) {};
				&& \node[label = {[xshift=-0.1cm]left:{23154}}] (23154) {};
				&& \node[label = {right:{14253}}] (14253) {}; \\
			&&&&&& \node[label = {[yshift=0.2cm]below:{13254}}] (13254) {}; \\
		};
	
	\draw(13254)--(13524);
	\draw(13254)--(31254);
	\draw(13254)--(13452);
	\draw(13254)--(15234);
	\draw(13254)--(23154);
	\draw(13254)--(14253);
	
	\draw(13524)--(31524);
	\draw(13524)--(13542);
	\draw(13524)--(15324);
	
	\draw(31254)--(31524);
	\draw(31254)--(31452);
	\draw(31254)--(32154);
	
	\draw(13452)--(13542);
	\draw(13452)--(31452);
	\draw(13452)--(14352);
	
	\draw(15234)--(15324);
	\draw(15234)--(25134);
	\draw(15234)--(15243);
	
	\draw(23154)--(32154);
	\draw(23154)--(25134);
	\draw(23154)--(24153);
	
	\draw(14253)--(14352);
	\draw(14253)--(15243);
	\draw(14253)--(24153);
	
	\draw(31524)--(31542);
	\draw(31524)--(35124);
	
	\draw(13542)--(31542);
	\draw(13542)--(15342);
	
	\draw(15324)--(35124);
	\draw(15324)--(15342);
	
	\draw(31452)--(31542);
	\draw(31452)--(34152);
	
	\draw(32154)--(35124);
	\draw(32154)--(34152);
	
	\draw(25134)--(35124);
	\draw(25134)--(25143);
	
	\draw(14352)--(15342);
	\draw(14352)--(34152);
	
	\draw(15243)--(25143);
	\draw(15243)--(15342);
	
	\draw(24153)--(34152);
	\draw(24153)--(25143);
	
	\draw(31542)--(35142);
	\draw(35124)--(35142);
	\draw(15342)--(35142);
	\draw(34152)--(35142);
	\draw(25143)--(35142);
		\end{tikzpicture}
		\caption{Bruhat interval $[13254, 35142]$.}\label{fig:13254-35142}
\end{figure}
\end{example}

We note that if $Q_{v,w}$ is a cube, then $[v,w]$ is Boolean and $Q_{v,w}$ is toric by Theorem~\ref{theo:3-6} and $Q_{v,w}$ is obviously simple.
\begin{corollary} \label{coro:3-8}
Any two of the following three statements imply that $Q_{v,w}$ is a cube and hence imply the remaining one:
\begin{enumerate}
\item $[v,w]$ is Boolean.
\item $Q_{v,w}$ is toric.
\item Either $v$ or $w$ is a simple vertex of $Q_{v,w}$.
\end{enumerate}
\end{corollary}
\begin{proof}
By Theorem~\ref{theo:3-6} and Proposition~\ref{prop:3-7}, it suffices to show that (1) and (3) imply (2).  Suppose that (1) and (3) hold.  Since $[v,w]$ is Boolean, the numbers of edges at $v$ and $w$ are equal to $\ell(w)-\ell(v)$.  On the other hand, since either $v$ or $w$ is a simple vertex of $Q_{v,w}$, the numbers of edges at $v$ and $w$ are equal to $\dim Q_{v,w}$.  Therefore $\dim Q_{v,w}=\ell(w)-\ell(v)$, that is, $Q_{v,w}$ is toric.
\end{proof}

\section{Product of Bruhat intervals}\label{sec:Product of Bruhat intervals}
In this section, we will show that there are infinitely many non-simple toric Bruhat interval polytopes (see~Proposition~\ref{prop:4-3}).

Let $r$ be a non-negative integer. To a pair $(x,y)\in \mathfrak{S}_p\times \mathfrak{S}_q$, we associate an element in $\mathfrak{S}_{p+q+r-1}$, denoted by $x*_r y$, as follows: express $x=s_{i_1}\cdots s_{i_k}\in \mathfrak{S}_p$, $y=s_{j_1}\cdots s_{j_\ell}\in\mathfrak{S}_q$, and define
\begin{equation} \label{eq:4-1}
x*_r y:=(s_{i_1}\cdots s_{i_k})(s_{j_1+p+r-1}\cdots s_{j_\ell+p+r-1})\in \mathfrak{S}_{p+q+r-1}.
\end{equation}
Since $i_1,\dots,i_k$ are less than or equal to $p-1$ while $j_1+p+r-1,\dots, j_\ell+p+r-1$ are greater than or equal to $p$, $x*_ry$ is well-defined, {that is}, independent of the expressions of~$x$ and~$y$ above. The expressions of~$x$ and~$y$ need not be reduced but if they are reduced, then the resulting expression of $x*_ry$ in \eqref{eq:4-1} is also reduced and hence
\begin{equation} \label{eq:4-2}
\ell(x*_ry)=\ell(x)+\ell(y).
\end{equation}
The following lemma would be obvious.

\begin{lemma} \label{lemm:4-1}
Let $x,x'\in \mathfrak{S}_p$ and $y,y'\in\mathfrak{S}_q$. Then
\begin{enumerate}
\item $x'*_r y'\le x*_r y$ if and only if $x'\le x$ and $y'\le y$,
\item $x'*_ry'\lessdot x*_ry$ if and only if $x'\lessdot x$ and $y'=y$ or $x'=x$ and $y'\lessdot y$.
\end{enumerate}
Moreover, if $z\le x*_r y$, then $z=x'*_r y'$ for some $x'\le x$ and $y'\le y$.
\end{lemma}

Suppose that $x'\le x$ and $y'\le y$. Then it follows from Lemma~\ref{lemm:4-1} that
\[
[x'*_ry',x*_ry]=\{ a*_rb\mid x'\le a\le x,\ y'\le b\le y\}.
\]
One can also see that
\begin{equation} \label{eq:4-3}
\dim Q_{x'*_r y',x*_ry}=\dim Q_{x',y'}+\dim Q_{x,y}
\end{equation}
using Theorem 4.6 in \cite{ts-wi15}.
Therefore, we have

\begin{corollary} \label{coro:4-2}
The poset structure of $[x'*_ry',x*_ry]$ is independent of $r$ \textup{(}$r\ge 0$\textup{)}. Moreover, $[x'*_ry',x*_ry]$ is Boolean if and only if both $[x',x]$ and $[y',y]$ are Boolean, and $Q_{x'*_ry',x*_ry}$ is toric if and only if both $Q_{x',x}$ and $Q_{y',y}$ are toric.
\end{corollary}

We set
\[
c(v,w):=\ell(w)-\ell(v)-\dim Q_{v,w}
\]
and call it the \emph{complexity} of the interval $[v,w]$ because it is the complexity of the torus action on the Richardson variety $X^v_w$.
By~\eqref{eq_dim_Qvw_and_length}, we get $c(v,w)\ge 0$ and the equality holds when $Q_{v,w}$ is toric by definition.  It follows from \eqref{eq:4-2} and \eqref{eq:4-3} that
\begin{equation*}
c(x'*_ry',x*_ry)=c(x',y')+c(x,y).
\end{equation*}

\begin{example}\label{exam:6-3}
\begin{enumerate}
	\item For $v = 1324$ and $w = 4231$ in Example~\ref{exam:3-1}(1), we get that
	\[
	c(v,w) = \ell(w)- \ell(v) - \dim Q_{v,w} = 5 - 1 - 3 = 1.
	\]
	Moreover, $[v,w]$ is a Boolean interval of length $4$.
	\item For $v = 1324$ and $w = 3412$ in Example~\ref{exam:3-1}(2), the Bruhat interval polytope $Q_{v,w}$ is toric of dimension~$3$, and we have that
	\[
	d(v) = d(w) = 4 = \dim Q_{v,w} + 1.
	\]
	\item For $v = 13254$ and $w = 35142$ in~Example~\ref{exam:3-2}, the Bruhat interval polytope $Q_{v,w}$ is toric of dimension~$4$, and we see that
	\[
	|d(v) - d(w)| = |6 -5| = 1.
	\]
\end{enumerate}
\end{example}

The following implies that there are infinitely many toric singular Bruhat interval polytopes.

\begin{proposition} \label{prop:4-3}
For any non-negative integer $k$,
\begin{enumerate}
\item there is a Boolean interval $[v,w]$ with $c(v,w)=k$,
\item there is a toric Bruhat interval polytope $Q_{v,w}$ such that $d(v)=d(w)=\dim Q_{v,w}+k$, and
\item there is a toric Bruhat interval polytope $Q_{v,w}$ with $|d(v)-d(w)|=k$.
\end{enumerate}
\end{proposition}

\begin{proof}
The case $k=0$ is realized by a cube $Q_{v,w}$, so we may assume $k\ge 1$.
The complexity $c(v,w)$, the degrees $d(v)$ and $d(w)$, and $\dim Q_{v,w}$ behave additively with respect to the product $*_r$ of a copy of $[v,w]$, so each statement respectively follows from {Example~\ref{exam:6-3}.}
\end{proof}

\section{Conditions on $v$ and $w$ for $Q_{v,w}$ to be toric}\label{sec:Conditions on v and w}
In this section, we first find some sufficient conditions on $v$ and $w$ for $Q_{v,w}$ to be toric, and then find a sufficient condition for such a toric Bruhat interval polytope $Q_{v,w}$ to be a cube.

It was shown in~\cite[{\S 5 and \S 6}]{ha-ho-ma-pa18} that $Q_{v,w}$ is toric (in fact, a cube) if $v=[a_1,\dots,a_{n-1},n]$ and $w=[n,a_1,\dots,a_{n-1}]$ or $v=[1,b_2,\dots,b_n]$ and $w=[b_2,\dots,b_n,1]$. In these cases,
\[
\begin{split}
&w=vs_{n-1}s_{n-2}\cdots s_1 \quad\text{and}\quad \ell(w)-\ell(v)=n-1,\\
&w=vs_1s_2\cdots s_{n-1} \quad\text{and}\quad \ell(w)-\ell(v)=n-1.
\end{split}
\]
These examples motivate us to study the following case:
\begin{equation} \label{eq:6-1}
\text{$w=vs_{j_1}s_{j_2}\cdots s_{j_m}$ where $\ell(w)-\ell(v)=m$ and $j_1,\dots,j_m$ are distinct.}
\end{equation}

\begin{proposition} \label{prop:6-1}
Suppose that $w=vs_{j_1}s_{j_2}\cdots s_{j_m}$  or  $w=s_{j_1}s_{j_2}\cdots s_{j_m}v$ with $\ell(w)-\ell(v)=m$. Then $j_1,\dots,j_m$ are distinct if and only if $Q_{v,w}$ is toric.
\end{proposition}

\begin{proof}
Since $w=s_{j_1}s_{j_2}\cdots s_{j_m}v$ means $w^{-1}=v^{-1}s_{j_1}s_{j_2}\cdots s_{j_m}$, it is enough to prove when $w=vs_{j_1}s_{j_2}\cdots s_{j_m}$ by Corollary~\ref{coro:combi-equiv}. Since $m=\ell(w)-\ell(v)$, we have $\ell(vs_{j_1}s_{j_2}\cdots s_{j_k})=\ell(v)+k$ for any $1\le k\le m$. This means that
\[
v\lessdot vs_{j_1} \lessdot vs_{j_1}s_{j_2}\lessdot \dots \lessdot vs_{j_1}s_{j_2}\cdots s_{j_m}=w
\]
is a maximal chain from $v$ to $w$, say $\mathcal{C}$.
Then it defines the graph $G^{\mathcal{C}}$ whose edge set is given by $\{\{{j_1},{j_1+1}\},\dots,\{{j_m},{j_m+1}\}\}$ (see~Section~\ref{sec:Properties of Bruhat interval polytopes}). Hence the number of connected components of~$G^\mathcal{C}$ is greater than or equal to $n-m$. Thus the dimension of $Q_{v,w}$ is less than or equal to $m$ by~\eqref{eq_dim_Qvw_Bvw}.
Notice that $G^\mathcal{C}$ has exactly $n-m$ components if and only if $j_1, j_2,\dots,j_m$ are distinct. Hence $Q_{v,w}$ is toric if and only if $j_1, j_2,\dots,j_m$ are distinct.
\end{proof}

Example~\ref{exam:3-1}(2) ($w=vs_3s_1s_2=vs_1s_3s_2$ where $v=s_2$) and Example~\ref{exam:3-2} ($w=vs_3s_4s_1s_2$ where $v=s_2s_4$) show that we cannot conclude that $Q_{v,w}$ is a cube in Proposition~\ref{prop:6-1}. We shall give a sufficient condition on $v$ and $w$ for $Q_{v,w}$ to be a cube. For that we prepare some notations.
For $p$ {and} $q$ {in} $[n-1]$, we set
\[
s(p,q)=\begin{cases} s_ps_{p+1}\cdots s_q \quad&\text{when $p\le q$},\\
s_ps_{p-1}\cdots s_q \quad &\text{when $p\ge q$}.
\end{cases}
\]
For each $s(p,q)$, we also set
\[
\bar{p}=\min\{p,q\},\quad \bar{q}=\max\{p,q\}.
\]
We note that if $j_1,\dots,j_m\in [n]$ are distinct, then we have a \emph{minimal} expression
\begin{equation} \label{eq:6-2}
s_{j_1}s_{j_2}\cdots s_{j_m}=s(p_1,q_1)s(p_2,q_2)\cdots s(p_r,q_r)
\end{equation}
where the intervals $[\bar{p}_1,\bar{q}_1],\dots, [\bar{p}_r,\bar{q}_r]$ are disjoint and $r$ is the minimum among such expressions.

\begin{example} \label{exam:6-1}
Here are examples of minimal expressions.
\begin{enumerate}
\item $s_1s_2\cdots s_{n-1}=s(1,n-1)$,\quad $s_{n-1}s_{n-2}\cdots s_1=s(n-1,1)$.

\item $s_1s_3s_8s_2s_4s_7s_6=s_3s_4s_1s_2s_8s_7s_6=s(3,4)s(1,2)s(8,6)$.

\item $s_2s_8s_4s_7s_1s_6=s_2s_1s_4s_8s_7s_6=s(2,1)s(4,4)s(8,6)$.
\end{enumerate}
\end{example}

We say that the product $s_{j_1}s_{j_2}\cdots s_{j_m}$ in \eqref{eq:6-2} is \emph{proper} if no two intervals among $[\bar{p}_1,\bar{q}_1],\dots, [\bar{p}_r,\bar{q}_r]$ are adjacent, in other words, the cycles defined by $s(p_1,q_1)$,$\ldots$, $s(p_r,q_r)$ are disjoint. For instance, (1) and (3) in Example~\ref{exam:6-1} are proper but (2) is not because the intervals $[3,4]$ and $[1,2]$ are adjacent.

\begin{proposition} \label{prop:6-2}
Suppose that $s_{j_1}s_{j_2}\cdots s_{j_m}$ is a proper minimal expression.
If $w=vs_{j_1}s_{j_2}\cdots s_{j_m}$ or $w=s_{j_1}s_{j_2}\cdots s_{j_m}v$ with $\ell(w)-\ell(v)=m$, then the Bruhat interval polytope $Q_{v,w}$ is a cube.
\end{proposition}
\begin{proof}
Since $w=s_{j_1}s_{j_2}\cdots s_{j_m}v$ means $w^{-1}=v^{-1}s_{j_1}s_{j_2}\cdots s_{j_m}$, it is enough to prove the proposition when $w=vs_{j_1}s_{j_2}\cdots s_{j_m}$ by Corollary~\ref{coro:cube-equiv}.
We know that $Q_{v,w}$ is toric by Proposition~\ref{prop:6-1} and hence every cover relation in $[v,w]$ gives an edge in $Q_{v,w}$ by Theorem~\ref{prop:3-2}. Thus it suffices to show that $v$ is covered by exactly $\ell(w)-\ell(v)$ elements in $[v,w]$ {by} Proposition~\ref{prop:3-7}.

Note that when $w=vs(1,n-1)$ or $vs(n-1,1)$, we know that $Q_{v,w}$ is an $(n-1)$-cube by~\cite[{\S 5}]{ha-ho-ma-pa18}. Therefore $v$ is covered by exactly $(n-1)$ elements in $[v,w]$ (we will show how to find those $(n-1)$ elements after the proof of the proposition).

Now we assume that $w=vs(p_1,q_1)\cdots s(p_r,q_r)$, where  $s(p_1,q_1)\cdots s(p_r,q_r)$ is a minimal expression with $r\ge 2$. For $u\in\mathfrak{S}_n$ and $1\le a<b\le n$, we denote the block from $u(a)$ to $u(b)$ in the one-line notation of~$u$ by $u({[a,b]})$. Since no two intervals among $[\bar{p}_1,\bar{q}_1],\dots,[\bar{p}_r,\bar{q}_r]$ are adjacent, we have
\[
\text{$w({[\bar{p}_i,\bar{q}_i+1]})=v({[\bar{p}_i,\bar{q}_i+1]})s(p_i,q_i)$\quad for each $i=1,2,\dots,r$.}
\]
Namely, on each block $[\bar{p}_i,\bar{q}_i+1]$, the situation is the same as the first case treated above. Therefore, $v$ is covered by exactly $\sum_{i=1}^r(\bar{q}_i+1-\bar{p}_i)$ elements in $[v,w]$. Here, $\sum_{i=1}^r(\bar{q}_i+1-\bar{p}_i)=\ell(w)-\ell(v)$, so this proves the {proposition}. To be more precise, we need the following observation. 
Any cover relation is obtained by a right transposition and we need to see that any right transposition on $v$ which gives a cover relation in $[v,w]$ is a transposition on some block $[\bar{p}_i,\bar{q}_i+1]$. But this is true because the partition $B_{v,w}$ of~$[n]$ is given by
\[
B_{v,w}=\bigcup_{i=1}^r[\bar{p}_i,\bar{q}_i+1]\cup \bigcup_{k\in [n]\backslash \bigcup_{i=1}^r[\bar{p}_i,\bar{q}_i+1]}\{k\}
\]
and a right transposition on $v$ which gives a cover relation in $[v,w]$ must preserve the partition $B_{v,w}$ since the partition $B_{v,w}$ is independent of a choice of maximal chains from $v$ to $w$ by~\cite[Corollary 4.8]{ts-wi15}.
\end{proof}

We briefly explain how to find the atoms of $[v,w]$ when $w=vs(1,n-1)$ or $vs(n-1,1)$ (see the proof of Proposition~\ref{prop:7-3}). When $w=vs(1,n-1)$ and $\ell(w)-\ell(v)=n-1$ (this means that $v(1)=1$). We write $v=v(1)v(2)\dots v(n)$ in one-line notation where $v(1)=1$. For each $v(j)$ for $2\le j\le n$, we find $v(i)$ such that $i<j$, $v(i)<v(j)$ and $v(k)>v(j)$ for any $i<k<j$, and interchange $v(i)$ and $v(j)$. The resulting element covers $v$ and one can check that it is in $[v,w]$.
The $(n-1)$ elements obtained in this way are the desired elements.

When $w=vs(n-1,1)$ and $\ell(w)-\ell(v)=n-1$ (this means that $v(n)=n$), the method to find elements in $[v,w]$ which cover $v$ is essentially same as above. We write $v=v(1)v(2)\dots v(n)$ in one-line notation where $v(n)=n$. For each $v(i)$ for $1\le i\le n-1$, we find $v(j)$ such that $i<j$, $v(i)<v(j)$ and $v(i)>v(k)$ for any $i<k<j$ and interchange $v(i)$ and $v(j)$. The resulting element covers $v$ and one can check that it is in $[v,w]$. The $(n-1)$ elements obtained {in} this way are the desired elements.

\begin{example} \label{exam:6-2}
	\begin{enumerate}
\item Take $v=14325$ and $w=vs(1,4)=43251$. In this case, the four elements in $[v,w]$ which cover $v$ are
\[
41325,\quad 34125,\quad 24315, \quad 14352.
\]
Although $15324$ and $14523$ cover $v$, they are not in $[v,w]$.

\item Take the same $v=14325$ as above but $w=vs(4,1)=51432$. In this case, the four elements in $[v,w]$ which cover $v$ are
\[
41325,\quad 15324,\quad 14523,\quad 14352.
\]
Although $34125$ and $24315$ cover $v$, they are not in $[v,w]$.
\item Take the same $v = 14325$ as above but $w = vs(1)s(4,3) = 41532$. There are three elements in $[v,w]$ cover $v$:
\[
41325 = 14325(1,2),\quad 14523 = 14325(3,5),\quad 14352 = 14325(4,5).
\]
We have that $\{1,2\} \subset [1,2]$ and $\{3,5\}, \{4,5\} \subset [3,5]$. Although $34125 = 14325(1,3)$, $24315 = 14325(1,4)$, and $15324 = 14325(2,5)$ which cover $v$, they are not in $[v,w]$. Here, one can see that none of the subsets $\{1,3\}, \{1,4\}, \{2,5\}$ of $[5]$ are contained in $[1,2]$ or $[3,5]$.
\end{enumerate}
\end{example}

Now we consider the converse of Proposition~\ref{prop:6-2}. If the product in \eqref{eq:6-2} is not proper, then
we may assume that the intervals $[\bar{p}_1,\bar{q}_1]$ and $[\bar{p}_{2},\bar{q}_{2}]$ in \eqref{eq:6-2} are adjacent by interchanging commuting factors if necessary.  We set $(a,b)=(p_1,q_1)$ and $(c,d)=(p_2,q_2)$. Then
\begin{equation} \label{eq:6-3}
w=vs(a,b)s(c,d)x \qquad \text{where}\quad x=s(p_3,q_3)\cdots s(p_r,q_r).
\end{equation}
The automorphism of $\mathfrak{S}_n$ given by conjugation by $w_0$ maps $s_i$ to $s_{n-i}$ and hence maps $s(a,b)s(c,d)$ to $s(n-a,n-b)s(n-c,n-d)$, so we may assume $a\le b$ without loss of generality.  

\begin{lemma} \label{lemm:6-3}
Assume that $a\leq b$ and $s(a,b)s(c,d)$ is minimal but not proper. Then
the following three cases occur:
\begin{enumerate}
\item $b+1=d$, $c>d$,
\item $a-1=c$, $c\ge d$, $a<b$,
\item $a-1=d$, $c<d$.
\end{enumerate}
\end{lemma}

\begin{proof}
Since the intervals $[a,b]$ and $[\bar{c},\bar{d}]$ are adjacent, we have $b+1=\bar{c}$ or $a-1=\bar{d}$. Then there are four possible cases.
\begin{enumerate}\setcounter{enumi}{-1}
\item If $b+1=\bar{c}=c$, then $c\le d$ and hence $s(a,b)s(c,d)=s(a,d)$ which contradicts the minimality of the expression $s(a,b)s(c,d)x$. Thus the case $b+1=c$ does not occur.
\item If $b+1=\bar{c}=d$, then $c\ge d$ but the case $c=d$ is excluded above. Hence $c>d$ which is case (1).
\item If $a-1=\bar{d}=c$, then $c\ge d$. Moreover, if $a=b$, then $s(a,b)s(c,d)=s(a,d)$ which contradicts the minimality of the expression $s(a,b)s(c,d)x$. Thus we obtain case (2).
\item If $a-1=\bar{d}=d$, then $c\le d$. Here the case $c=d$ is included in case (2), so we obtain case~(3).
\end{enumerate}
Therefore, the only three cases occurs.
\end{proof}

The following lemma shows that if a minimal expression $s(a,b)s(c,d)$ is not proper, then there exists a permutation $v$ such that $Q_{v,vs(a,b)s(c,d)}$ is not a cube.
\begin{lemma} \label{lemm:6-4}
Let $v=s_d$ for cases {\rm(1)} and {\rm(3)} in Lemma~\ref{lemm:6-3} and $v=s(a+1,b)s_c$ for case {\rm(2)} in Lemma~\ref{lemm:6-3}. Then {the Bruhat interval} $[v,w]$ for $w$ in \eqref{eq:6-3} has $\ell(w)-\ell(v)+1$ coatoms {\rm(}so $Q_{v,w}$ is not a cube although it is toric{\rm)}.
\end{lemma}

\begin{proof}
We note that any coatom of $[v,w]$ is obtained by removing an element from a reduced expression of~$w$ such that the resulting expression is reduced and contains a reduced expression of~$v$.

For case (1) in Lemma~\ref{lemm:6-3}, we have the reduced expression of~$w$:
\begin{equation} \label{eq:6-4}
w=vs(a,b)s(c,d)x=s_d(s_as_{a+1}\cdots s_{b-1}s_b)(s_cs_{c-1}\cdots s_{d+1}s_d)x.
\end{equation}
As remarked above, any coatom is obtained by removing an element from \eqref{eq:6-4}.
In \eqref{eq:6-4}, only $s_d$ appears twice, the others appear only once, and the elements $s_b=s_{d-1}$ and $s_{d+1}$ which do not commute with $s_d$ appear between the two {$s_{d}$'s} in \eqref{eq:6-4}. Noting these, one can see that removing any element from \eqref{eq:6-4} produces a reduced expression and the resulting expression contains $v=s_d$. Therefore $[v,w]$ has exactly $\ell(w)-\ell(v)+1$ coatoms.

A similar argument works for case (3). In this case we have
\begin{equation} \label{eq:6-5}
w=vs(a,b)s(c,d)x=s_d(s_as_{a+1}\cdots s_{b-1}s_b)(s_cs_{c+1}\cdots s_{d-1}s_d)x
\end{equation}
and only $s_d$ appears twice and the elements $s_a=s_{d+1}$ and $s_{d-1}$ which do not commute with $s_d$ appear between the two $s_d$'s. Therefore, removing any element from \eqref{eq:6-5} produces a reduced expression and the resulting expression contains $v=s_d$, so $[v,w]$ has exactly $\ell(w)-\ell(v)+1$ coatoms.

As for case (2), the situation is slightly different from the above two cases. In case (2) we have
\begin{equation} \label{eq:6-6}
\begin{split}
w&=vs(a,b)s(c,d)x\\
&=(s_{a+1}s_{a+2}\cdots s_{b-1}s_bs_c)(\underline{s_as_{a+1}s_{a+2}\cdots s_{b-1}}s_b)(s_cs_{c-1}\cdots s_{d+1}s_d)x.
\end{split}
\end{equation}
In this case, $s_{a+1}, s_{a+2},\dots, s_{b-1}, s_b, s_c(=s_{a-1})$ appear twice and the others appear once in \eqref{eq:6-6}. One can see that removing any element from the underlined product in \eqref{eq:6-6} does not produce a reduced expression. For instance, if we remove $s_a$, then $s_c$ commute with all the elements between the two $s_c$ in \eqref{eq:6-6}; so the resulting expression is not reduced. If we remove $s_{a+1}$ in the underlined product, then \eqref{eq:6-6} turns into
\[
s_cs_{a+1}s_a(s_{a+2}\cdots s_{b-1}s_b)(s_{a+2}\cdots s_{b-1}s_b)(s_cs_{c-1}\cdots s_{d+1}s_d)x
\]
since $s_c(=s_{a-1})$ commutes with all the elements $s_k$ for $a+1\le k\le b$ and $s_a$ commutes with all the elements $s_\ell$ for $a+2\le \ell\le b$. The above expression is not reduced because the product $(s_{a+2}\cdots s_{b-1}s_b)(s_{a+2}\cdots s_{b-1}s_b)$ can be reduced. A similar observation applies when we remove one of the other elements in the underlined product in \eqref{eq:6-6}.

On the other hand, removing an element not in the underlined product produces a reduced expression (for that the existence of $s_a$ in \eqref{eq:6-6} is important).
One can also see that the elements obtained in this way
contains the reduced expression $s_{a+1}s_{a+2}\cdots s_{b-1}s_bs_c$ of $v$ and there are $\ell(w)-\ell(v)+1$ number of such elements, so $[v,w]$ has exactly $\ell(w)-\ell(v)+1$ coatoms in this case too.
\end{proof}

Proposition~\ref{prop:6-2} and the two lemmas above say that the converse of Proposition~\ref{prop:6-2} is true.
\begin{corollary} \label{coro:6-5}
The Bruhat interval polytope $Q_{v,w}$ is a cube for any $v$ and $w$ in~\eqref{eq:6-1} if and only if the product $s_{j_1} s_{j_2} \cdots s_{j_m}$ is proper.
\end{corollary}

\begin{example}
	Suppose that $s_{j_1} s_{j_2} s_{j_3} = s_3s_1s_2$.
	Then the Bruhat interval polytope $Q_{v,w}$ is a cube when $v = e$ and $w = s_3s_1s_2$ since $s_3, s_1, s_2$ are distinct. On the other hand, for $v = s_2$ and $w = s_2s_3s_1s_2$, the Bruhat interval polytope $Q_{v,w}$ is not a cube by Example~\ref{exam:3-1}(2). Indeed, the product $s_3s_1s_2 = s(3,3) s(1,2)$ is not proper. 
\end{example}

Below is another sufficient condition on $v$ and $w$ for $Q_{v,w}$ to be a cube.

\begin{proposition} \label{prop:6-6}
Suppose that $w$ is a product of distinct simple reflections {\rm(}equivalently, $w$ avoiding the patterns $3412$ and $321$ by Tenner \cite{tenn07}{\rm)}. Then $Q_{v,w}$ is a cube for any $v$ and $w$ such that $v< w$. In particular, if $v$ and $w$ are as in \eqref{eq:6-1} and $v$ has a reduced expression $s_{i_1}s_{i_2}\cdots s_{i_\ell}$ such that $i_1,\dots,i_\ell,j_1,\dots,j_m$ are all distinct, then $Q_{v,w}$ is a cube.
\end{proposition}

\begin{proof}
If $w$ is a product of distinct simple reflections, then $Q_{{e},w}$ is a cube. Since $Q_{{e},w}$ is in particular toric, $Q_{v,w}$ is a face of $Q_{{e},w}$ for any $v$ such that $v<w$ by {Theorem}~\ref{prop:3-2}. Therefore $Q_{v,w}$ is also a cube.
\end{proof}

So far, we have studied whether $Q_{v,w}$ is toric when there exist reduced expressions $r(v)$ and $r(w)$ for $v$ and $w$ such that the subword $r(w)\setminus r(v)$ of $r(w)$ is a product of distinct simple transpositions. Unfortunately, there is an example that $Q_{v,w}$ is a cube even though there are no reduced expressions for $v$ and $w$ such that $r(w)\setminus r(v)$ is distinct.

\begin{example}\label{ex:not-distinct-cube}
  Let $v=1243$ and $w=3412$. Then $v=s_3$ and $w=s_2s_3s_1s_2=s_2s_1s_3s_2$. Hence there are no reduced expressions of $v$ and $w$ such that $r(w)\setminus r(v)$ is a product of distinct simple transpositions. But one can check that the Bruhat interval polytope $Q_{v,w}$ is a $3$-cube.
\end{example}
Therefore, it seems difficult to characterize $v$ and $w$ for which $Q_{v,w}$ is toric or combinatorially equivalent to a cube.
\section{Finding all coatoms in some special cases}\label{sec:Finding all coatoms}

In this section, we will find all coatoms in some special cases. We first find a necessary and sufficient condition for ${w(i,j)}$ to be a coatom of $[v,w]$ when $w = vs(1,n-1)$, and then conclude that there are exactly $(n-1)$ elements in $[v,vs(1,n-1)]$ without using the result in~\cite[{\S 5}]{ha-ho-ma-pa18}. After that, we consider the case $w=vs(a,b)s(c,d)$ where $s(a,b)s(c,d)$ is minimal. We give a necessary and sufficient condition for ${w(i,j)}$ to be a coatom of $[v,w]$ and then describe when $Q_{v,w}$ is a cube.  Note that finding atoms is essentially same as finding coatoms because multiplication by the longest element $w_0$ reverses the Bruhat order.

The following lemma is obvious but plays a role in our argument.

\begin{lemma} \label{lemm:7-1}
Let $c_1,\dots,c_d$ be distinct positive integers.
\begin{enumerate}
\item Let $a$ and $b$ be positive integers different from any~$c_i$. Then
\[
\{ c_1,\dots,c_d,a\}\!\uparrow\ \le \{ c_1,\dots,c_d,b\}\!\uparrow\quad \Longleftrightarrow\quad a\le b.
\]
\item Let $a_1$ and $a_2$ (also $b_1,b_2$) be distinct positive integers different from any~$c_i$. Then
\[
\{ c_1,\dots,c_d,a_1,a_2\}\!\uparrow\ \le \{ c_1,\dots,c_d,b_1,b_2\}\!\uparrow\quad \Longleftrightarrow\quad
\{a_1,a_2\}\!\uparrow\ \le \{b_1,b_2\}\!\uparrow.
\]
\end{enumerate}
\end{lemma}

\medskip

Note that if $\{a_1,a_2\}\cap \{b_1,b_2\}\not=\emptyset$, then {\rm (2)} reduces to {\rm (1)} in the lemma above.
We prepare one more lemma.

\begin{lemma} \label{lemm:7-2}
Let $v\le w$ and let $t_{i,j}$ $(1\le i<j\le n)$ be a transposition. Then $v\le wt_{i,j}$ if and only if
\begin{equation*}
\{v(1),v(2),\dots,v(p)\}\!\uparrow\, \le \{wt_{i,j}(1),wt_{i,j}(2),\dots,wt_{i,j}(p)\}\!\uparrow \ \text{for every $i\le  p<j$}.
\end{equation*}
\end{lemma}

\begin{proof}
The condition $v\le wt_{i,j}$ is equivalent to
\begin{equation*}
\{v(1),v(2),\dots,v(p)\}\!\uparrow\, \le \{wt_{i,j}(1),wt_{i,j}(2),\dots,wt_{i,j}(p)\}\!\uparrow \ \text{for every $1\le p<n$}. \end{equation*}
For $p<i$ or $p\ge j$, we have
$$\{wt_{i,j}(1),wt_{i,j}(2),\dots,wt_{i,j}(p)\}=\{w(1),w(2),\dots,w(p)\}.$$
Since $v\le w$, this shows that the inequality above holds for $p<i$ or $p\ge j$, proving the lemma.
\end{proof}

We set
\[
V(p)=\{v(1),\dots,v(p)\},\qquad W_{i,j}(p)=\{ wt_{i,j}(1),\dots,wt_{i,j}(p)\}.
\]
We also introduce the following notation: if $A$ and $B$ are sets and $C$ is a subset of $A\cap B$, then we write
\[
(A,B)\equiv (A\backslash C,B\backslash C).
\]
We will apply this notation to $V(p)$ and $W_{i,j}(p)$ later.

In this section we investigate the following case
\begin{equation} \label{eq:7-1}
w=vs_1s_2\cdots s_{n-1} \quad\text{equivalently}\quad v=ws_{n-1}\cdots s_2s_1
\end{equation}
where $\ell(w)-\ell(v)=n-1$.
(A similar argument works when $w=vs_{n-1}s_{n-2}\cdots s_1$.) Since $\ell(w)-\ell(v)=n-1$, it follows from \eqref{eq:7-1} that
\begin{equation*} 
w(n)<w(1),\dots,w(n-1).
\end{equation*}
Note that these inequalities imply that $w(n)=1$.

\begin{proposition} \label{prop:7-3}
Let $v$ and $w$ be as in \eqref{eq:7-1}. Then $wt_{i,j}$ is a coatom of $[v,w]$ if and only if $w(i)>w(j)$ and $w(p)>w(i)$ for every $i< p<j$. Moreover, there are exactly $n-1$ coatoms in $[v,w]$ and hence $Q_{v,w}$ is a cube.
\end{proposition}

\begin{proof}
We note that $\ell(wt_{i,j})=\ell(w)-1$ if and only if $w(i)>w(j)$ and $w(p)\notin [w(j),w(i)]$ for every $i< p<j$. We shall show that under this situation, the condition $v\le wt_{i,j}$ is equivalent to the condition $w(i)<w(p)$ for every $i< p<j$.

It follows from \eqref{eq:7-1} that we have
\[
\arraycolsep=1.4pt
\begin{array}{rcccccccccccc}
wt_{i,j}=&w(1)& w(2) &\ldots &w(i-1)&\cfbox{red}{$w(j)$}&w(i+1)&\ldots &w(j-1)&\cfbox{red}{$w(i)$}&w(j+1)&\ldots &w(n),\\
v=&w(n)&w(1)& \ldots &w(i-2)&w(i-1)&w(i)&\ldots &w(j-2)&w(j-1)&w(j)&\ldots &w(n-1).
\end{array}
\]
Therefore, we have
\[
(V(p),W_{i,j}(p)) \equiv 
\begin{cases} 
{(\{w(n)\}, \{w(j)\})} & {\text{for } p = i,} \\
(\{w(n),w(i)\}, \{w(j),w(p)\}) & \text{for $i < p<j$}.
\end{cases}
\]
Then
\[
\begin{split}
V(p)\!\uparrow\ \le W_{i,j}(p)\!\uparrow\ &\Longleftrightarrow\ w(n) \leq w(j) \quad \text{ for } p = i,\\
V(p)\!\uparrow\ \le W_{i,j}(p)\!\uparrow\ &\Longleftrightarrow\ \{w(n),w(i)\}\uparrow\ \le \{w(j),w(p)\}\!\uparrow
\quad \text{ for }i < p < j
\end{split}
\]
by Lemma~\ref{lemm:7-1}. Since $1=w(n)<w(j)<w(i)$, it follows from Lemma~\ref{lemm:7-2} that $wt_{i,j}$ is a coatom of~$[v,w]$ if and only if $w(i)<w(p)$ for every $i< p<j$, proving the former statement of the proposition.

For each $1\le i<n$, there exists $j$ satisfying the condition in the former statement since $w(n)=1$, and such $j$ is unique for each $i$. This proves the latter statement in the proposition.
\end{proof}

In the remainder of this section we will treat the case
\[
w=vs(a,b)s(c,d) \qquad\text{($s(a,b)s(c,d)\not=s(\bar{a},\bar{d})$)},
\]
where $[\bar{a},\bar{b}]\cup [\bar{c},\bar{d}]=[1,n-1]$ and $\ell(w)-\ell(v)=n-1$. The conjugation on $\mathfrak{S}_n$ by $w_0$ maps $s_i$ to $s_{n-i}$, so it suffices to consider the case where $a<b$. There are three cases:
\begin{enumerate}
\item[I\,:] $(a,b)=(1,k-1)$, $(c,d)=(n-1,k)$,
\item[II\,:] $(a,b)=(k,n-1)$, $(c,d)=(k-1,1)$,
\item[III\,:] $(a,b)=(k,n-1),$ $(c,d)=(1,k-1)$,
\end{enumerate}
where $n\ge 4$ and $2\le k\le n-2$.

In the following, we assume that $\ell(wt_{i,j})=\ell(w)-1$, so
\begin{equation*}
\text{$w(i)>w(j)$ and $w(p)\notin [w(j),w(i)]$ for every $i< p<j$.}
\end{equation*}
We keep in mind that $w(i)>w(j)$ throughout this section unless otherwise stated.
We shall observe that the condition $v\le wt_{i,j}$ gives stronger conditions than the above. By Lemma~\ref{lemm:7-2},
it suffices to check
\begin{equation} \label{eq:8-1}
V(p)\!\uparrow\ \le W_{i,j}(p)\!\uparrow \quad\text{for $i\le p<j$}.
\end{equation}

\medskip
{\bf Case I.} In this case we have
\begin{equation*}
\begin{split}
w&=v(s_1\cdots s_{k-1})(s_{n-1}\cdots s_k),\quad\text{equivalently}\\
v&=w(s_k\cdots s_{n-1})(s_{k-1}\cdots s_1).
\end{split}
\end{equation*}
Since $\ell(w)-\ell(v)=n-1$, we have
\begin{equation} \label{eq:8-2}
\begin{split}
&w(k+1),w(k+2),\dots,w(n)<w(k),\\
&w(k+1)<w(1),\dots,w(k-1).
\end{split}
\end{equation}

\begin{proposition} \label{prop:8-1}
In Case I, $wt_{i,j}$ is a coatom of $[v,w]$ if and only if one of the following is satisfied:
\begin{enumerate}
\item If $1\le i<j\le k+1$, then $w(i)<w(p)$ for every $i< p<j$.
\item If $k\le i<j\le n$, then $w(p)<w(j)$ for every $i< p<j$.
\item If $1\le i< k$ and $k+1< j\le n$, then $w(i)<w(p)$ for every $i< p\le k$ and $w(p)<w(j)$ for every $k<  p<j$.
\end{enumerate}
Moreover, $Q_{v,w}$ is a cube in Case I if and only if there is no pair $(i,j)$ in {\rm(3)}.
\end{proposition}

\begin{proof}
(1) In this case we have
\[
{\scriptsize
	\arraycolsep = 1.4pt
\begin{array}{rccccccccccccccc}
wt_{i,j}=& w(1)& w(2)& \dots &w(i-1)&\cfbox{red}{$w(j)$}& w(i+1)& \dots& w(j-1)& \cfbox{red}{$w(i)$}& \dots& w(k)&w(k+1)&\dots &w(n-1)&w(n),\\
v=& w(k+1) &w(1) &\dots& w(i-2)&w(i-1)&w(i)& \ldots&w(j-2)&w(j-1)&\dots &w(k-1)&w(k+2)&\dots&w(n)&w(k).
\end{array}
}
\]

where $w(i-1)$ in the line of~$v$ for $i=1$ is understood to be $w(k+1)$. Therefore
\begin{equation*}
(V(p),W_{i,j}(p))\equiv 
\begin{cases}
(\{w(k+1)\}, \{w(j)\}) & {\text{for } p = i,} \\
(\{w(k+1),w(i)\}, \{w(j),w(p)\}) &\text{for $i < p< j$}.
\end{cases}
\end{equation*}
Here $w(k+1)\le w(j)<w(i)$ by \eqref{eq:8-2} because $j\le k+1$. Therefore, it follows from Lemma~\ref{lemm:7-1} that \eqref{eq:8-1} is equivalent to $w(i)< w(p)$ for $i< p< j$, proving (1).

(2) In this case we have
\[
{\scriptsize
	\arraycolsep = 1.4pt
	\begin{array}{rccccccccccccccc}
wt_{i,j}=&w(1)&w(2) &\dots& w(k)& w(k+1)&\dots &w(i-1)&\cfbox{red}{$w(j)$}&w(i+1)& \dots& w(j-1)&\cfbox{red}{$w(i)$}&\dots& w(n-1)&w(n),\\
v= &w(k+1) &w(1)& \dots& w(k-1)&w(k+2)&\dots& w(i)&w(i+1)&w(i+2)&\dots&w(j)&w(j+1)&\dots &w(n)&w(k).
\end{array}}
\]
Therefore
\begin{equation*}
(V(p),W_{i,j}(p))\equiv ( \{w(p+1),w(i)\}, \{w(j),w(k)\}) \quad\text{for $i\le p< j$}.
\end{equation*}
Here $w(j)<w(i)\le w(k)$ by \eqref{eq:8-2} because $k\le i$. Therefore, it follows from Lemma~\ref{lemm:7-1} that \eqref{eq:8-1} is equivalent to $w(p+1)\le w(j)$ for $i\le p< j$, proving (2).

(3) In this case we have
\[
{\scriptsize
		\arraycolsep = 1.4pt
		\begin{array}{rccccccccccccccc}
wt_{i,j}=&w(1)& w(2) &\dots& w(i-1)&\cfbox{red}{$w(j)$}&w(i+1)&\dots& w(k)& w(k+1)&\dots &w(j-1)& \cfbox{red}{$w(i)$}&\dots& w(n-1)&w(n),\\
v= &w(k+1)& w(1)& \dots& w(i-2)&w(i-1)&w(i)&\dots &w(k-1)&w(k+2)&\dots &w(j-2)&w(j-1)&\dots& w(n)&w(k).
\end{array}}
\]
Therefore
\begin{equation*}
(V(p),W_{i,j}(p))\equiv \begin{cases} 
{(\{w(k+1)\}, \{w(j)\})} & {\text{for } p = i,} \\
(\{w(k+1),w(i)\}, \{w(j),w(p)\}) &\text{for $i< p\le k$},\\
(\{w(p+1),w(i)\}, \{w(j),w(k)\}) &\text{for $k< p<j$.}
\end{cases}
\end{equation*}

First we treat the case for $i\le p\le k$. Note that $w(k+1)<w(i)$ by \eqref{eq:8-2} because $i\le k$. Therefore, it follows from Lemma~\ref{lemm:7-1} that \eqref{eq:8-1} is equivalent to
\begin{equation} \label{eq:8-3}
\text{$w(k+1)\le w(j)$ and $w(i)\le w(p)$ for $i\le p\le k$.}
\end{equation}
As for the case {when} $k<p<j$, note that $w(j)<w(k)$ by \eqref{eq:8-2} because $k+1< j$. Therefore, it follows from Lemma~\ref{lemm:7-1} that \eqref{eq:8-1} is equivalent to
\begin{equation} \label{eq:8-4}
\text{$w(p+1)\le w(j)$ for $k<p<j$ and $w(i)\le w(k)$.}
\end{equation}
Inequalities \eqref{eq:8-3} and \eqref{eq:8-4} prove case (3).

In case (1), the latter inequalities in \eqref{eq:8-2} ensure the existence of the desired $j$ for each $i$ and such $j$ is unique for each $i$; so there are exactly $k$ desired pairs $(i,j)$ in case (1). The same is true for case (2) with the role of~$i$ and~$j$ interchanged. Namely, for each $j$ there exists a unique desired $i$ where the existence of~$i$ for the~$j$ is ensured by the former inequalities in \eqref{eq:8-2}; so there are exactly $n-k$ desired pairs $(i,j)$ in case (2). However, cases (1) and (2) have one overlap, that is, the case $(i,j)=(k,k+1)$ and this case satisfies the required condition. Therefore, we obtain exactly $n-1$ coatoms of $[v,w]$ from cases (1) and (2).
This proves the last statement in the proposition.
\end{proof}

\begin{remark}
The number of the pairs $(i,j)$ in case (3) is at most $(k-1)(n-k-1)$ and the following example attains the maximum:
\[
{
	\arraycolsep = 1.4pt
\begin{array}{rcrrrrrrrrrr}
w&= [&n-k+1,&n-k+2,&\ldots,&n,&1,&2,&\ldots,&n-k-1,&n-k&],\\
v&=[&1,&n-k+1,&\ldots,&n-1,&2,&3,&\ldots,&n-k,&n&].
\end{array}
}
\]

Therefore, the Bruhat interval $[v,w]$, which is of length $n-1$, has 
\[
(n-1)+(k-1)(n-k-1)=k(n-k) 
\]
{many} coatoms. Note that $k(n-k) \le \lfloor n^2/4\rfloor$, and the equality is attained when $k=\lfloor n/2\rfloor$. It is shown in \cite[{Theorem in \S 1}]{koba11} that the number of coatoms of any Bruhat interval of length $n-1$ is at most $\lfloor n^2/4\rfloor$ and that the maximum can be attained by the above example.
\end{remark}

Two types of $(V(p),W_{i,j}(p))$ appear for case (3) in the above proof and each appears for cases (1) and (2) respectively. This implies that it suffices to treat case (3) essentially.

\medskip

{\bf Case II.} In this case we have
\begin{equation*}
\begin{split}
w&=v(s_k\cdots s_{n-1})(s_{k-1}\cdots s_1),\quad\text{equivalently}\\
v&=w(s_1\cdots s_{k-1})(s_{n-1}\cdots s_k).
\end{split}
\end{equation*}
Since $\ell(w)-\ell(v)=n-1$, we have
\begin{equation} \label{eq:8-5}
\begin{split}
&w(2),w(3),\dots,w(k)<w(1),\\
&w(n)<w(1),w(k+1),\dots,w(n-1).
\end{split}
\end{equation}

\begin{proposition} \label{prop:8-2}
In Case II, $wt_{i,j}$ is a coatom of $[v,w]$ if and only if one of the following is satisfied:
\begin{enumerate}
\item If $1\le i<j\le k$, then $w(p)<w(j)$ for every $i< p<j$.
\item If $k< i<j\le n$, then $w(i)<w(p)$ for every $i< p<j$.
\item If $1\le i\le k< j\le n$, then $w(p)<w(j)$ for every $i< p\le k$ and $w(i)<w(p)$ for every $k< p<j$.
\end{enumerate}
Moreover, $Q_{v,w}$ is a cube in Case II if and only if there is only one pair $(i,j)$ in {\rm(3)}.
\end{proposition}

\begin{proof}
Suppose that $1\le i< k$ and $k+1< j\le n$ . Then we have
\[
{\scriptsize
		\arraycolsep = 1.4pt
		\begin{array}{rcccccccccccccc}
wt_{i,j}= &w(1)&\dots &w(i-1)&\cfbox{red}{$w(j)$}&\dots& w(k-1)&w(k)&w(k+1)&w(k+2) &\dots& w(j-1)&\cfbox{red}{$w(i)$}& \dots &w(n),\\
v=& w(2)& \dots &w(i)&w(i+1)&\dots & w(k)& w(n)& w(1)& w(k+1) &\dots &w(j-2)&w(j-1)&\dots &w(n-1).
\end{array}}
\]
Therefore
\begin{equation} \label{eq:8-6}
(V(p),W_{i,j}(p))\equiv \begin{cases} (\{w(p+1),w(i)\}, \{w(j),w(1)\}) \quad&\text{for $i\le p< k$},\\
( \{w(n),w(i)\},\{w(j),w(1)\}) \quad&\text{for $p=k$},\\
( \{w(n),w(i)\}, \{w(j),w(p)\}) \quad&\text{for $k<p<j$}.
\end{cases}
\end{equation}
Here $w(n)\le w(j)<w(i)\le w(1)$ by \eqref{eq:8-5} because $1\le i<k$ and $k+1<j\le n$. Therefore, it follows from Lemma~\ref{lemm:7-1} that \eqref{eq:8-1} is equivalent to
\[
\begin{split}
w(p+1)<w(j) \quad&\text{for $i\le p<k$},\\
w(i)<w(p)\quad &\text{for $k<p<j$},
\end{split}
\]
proving the assertion when $1\le i<k$ and $k+1<j\le n$.

One can see that the same argument works for the remaining cases with a little modification. For instance, when $1\le i<j\le k$ (i.e., case (1) in the proposition), only the first type in \eqref{eq:8-6} occurs and when $k<i<j\le n$ (i.e., case (2) in the proposition), only the third type in \eqref{eq:8-6} occurs. The cases where $i=k$ or $j=k+1$ in case (3) remain. When $i=k$ and $j=k+1$, only the second type in \eqref{eq:8-6} appears and when $i=k$ and $k+1<j\le n$, the second and third types in \eqref{eq:8-6} appear and when $1\le i<k$ and $j=k+1$, the first and second types in \eqref{eq:8-6} appear.

In case (1), the former inequalities in \eqref{eq:8-5} ensure the existence of the desired $i$ for each $j$ and such $i$ is unique for the $j$; so there are exactly $k-1$ pairs $(i,j)$ in case (1). The same is true for case (2) with the role of~$i$ and~$j$ interchanged. Namely, for each $i$ there exists a unique desired $j$ where the existence of~$j$ for the $i$ is ensured by the latter inequalities in \eqref{eq:8-5}; so there are exactly $n-k-1$ pairs $(i,j)$ in case (2). Therefore, we obtain exactly $n-2$ coatoms from cases (1) and (2). This proves the last statement in the proposition.
\end{proof}

\begin{remark}
If $w(k)>w(k+1)$, then there is only one pair $(i,j)$ in case (3), that is $(i,j)=(k,k+1)$. Therefore, $Q_{v,w}$ is a cube in this case.
If $w(k)<w(k+1)$, then it happens that there are more than one pair $(i,j)$ in case (3) but the number of those pairs is at most $\min\{k,n-k\}$ because for each $i$, the desired~$j$ is unique if it exists and vice versa. The following examples attain the maximum $\min\{k,n-k\}$: \newline
when $k> n-k$,
\[
w=[n,n-2,\dots,n-2(n-k),2k-n-1,2k-n-2,\dots,1,n-1,n-3,\dots,n-1-2(n-k-1)]
\]
where $w(k)=1$, and when $k\le n-k$,
\[
w=[n,n-2,\dots,n-2(k-1),n-1,n-3,\dots,n-1-2(k-1),n-2k,n-2k-1,\dots,1]
\]
where $w(n)=1$. For instance
\[
\begin{split}
&w=[10,8,6,4,3,2,1,9,7,5] \qquad \text{when $(n,k)=(10,7)$},\\
&w=[10,8,6,4,2,9,7,5,3,1] \qquad \text{when $(n,k)=(10,5)$},\\
&w=[10,8,6,9,7,5,4,3,2,1] \qquad \text{when $(n,k)=(10,3)$}.
\end{split}
\]
\end{remark}

\medskip
{\bf Case III.} In this case we have
\begin{equation*}
\begin{split}
w&=v(s_k\cdots s_{n-1})(s_{1}\cdots s_{k-1}),\quad\text{equivalently}\\
v&=w(s_{k-1}\cdots s_{1})(s_{n-1}\cdots s_k).
\end{split}
\end{equation*}
Since $\ell(w)-\ell(v)=n-1$, we have
\begin{equation} \label{eq:8-7}
\begin{split}
&w(k)<w(1),\dots, w(k-1),\\
&w(n)<w(k-1),w(k+1),\dots,w(n-1).
\end{split}
\end{equation}
We note that $w(k)=1$ or $w(n)=1$. Indeed, if $w(k)<w(n)$ (respectively, $w(n)<w(k)$), then it follows from \eqref{eq:8-7} that $w(k)=1$ (respectively, $w(n)=1$).

\begin{proposition} \label{prop:8-3}
In Case III, $wt_{i,j}$ is a coatom of $[v,w]$ if and only if one of the following is satisfied:
\begin{enumerate}
\item If $1\le i<j\le k$ or $k< i<j\le n$, then $w(i)<w(p)$ for every $i< p<j$,
\item If $i=k<j\le n$, then $w(i)<w(p)$ for every $i< p<j$ and $w(n)=1$,
\item If $1\le i< k<j\le n$, then $w(i)<w(p)$ for every $i< p< j$ with $p\not=k$ and $w(k)<w(j)$.
\end{enumerate}
Moreover, $Q_{v,w}$ is a cube in Case III if and only if there is no {\rm(}respectively, only one{\rm)} pair $(i,j)$ in {\rm(3)} when $w(n)=1$ {\rm(}respectively, $w(k)=1${\rm)}.
\end{proposition}

\begin{proof}
Suppose that $1\le i<k$ and $k+1<j\le n$. Then we have
\[
{\scriptsize
		\arraycolsep = 1.4pt
		\begin{array}{rcccccccccccccccc}
wt_{i,j}= &w(1)&w(2)&\dots& w(i-1)&\cfbox{red}{$w(j)$}&w(i+1)&\dots& w(k-1)&w(k)&w(k+1)&w(k+2) &\dots &w(j-1)&\cfbox{red}{$w(i)$}& \dots &w(n),\\
v= &w(k)&w(1)& \dots &w(i-2)&w(i-1)&w(i)& \dots& w(k-2)&w(n)&w(k-1)&w(k+1)& \dots& w(j-2)&w(j-1)&\dots &w(n-1).
\end{array}}
\]
Therefore
\begin{equation*} \label{eq:8-8}
(V(p),W_{i,j}(p))\equiv 
\begin{cases} 
{(\{w(k)\}, \{w(j)\})} & {\text{for } p = i,} \\
(\{w(k),w(i)\}, \{w(j),w(p)\}) \quad&\text{for $i< p< k$},\\
(\{w(n),w(i)\}, \{w(j),w(k-1)\}) \quad&\text{for $p=k$},\\
(\{w(n),w(i)\}, \{w(j),w(p)\}) \quad&\text{for $k<p<j$}.
\end{cases}
\end{equation*}
Here $w(k)<w(i)$ and $w(n)\le w(j)<w(i)$ by \eqref{eq:8-7} because $i<k$ and $k+1<j\le n$. Therefore, it follows from Lemma~\ref{lemm:7-1} that \eqref{eq:8-1} is equivalent to
\[
\begin{split}
&w(k)<w(j),\\
&w(i)<w(p) \quad\text{for $i< p<k$ or $k<p<j$},
\end{split}
\]
proving the assertion when $1\le i<k$ and $k+1<j\le n$.

Similarly to the proof of Proposition~\ref{prop:8-2}, one can see that the same argument works for the remaining cases with a little modification.

For each $1\le i<k$ (respectively, $k< i< n$), the desired $j$ in case (1) is unique and the existence of such $j$ is ensured by the former (respectively, latter) inequalities in \eqref{eq:8-7}; so we obtain $n-2$ coatoms from case (1).
Since the desired $j$ in case (2) is also unique, we obtain one coatom from case (2) when $w(n)=1$ and none when $w(k)=1$. Therefore, we obtain $n-1$ (respectively, $n-2$) coatoms from cases (1) and (2) when $w(n)=1$ (respectively, $w(k)=1$).
This proves the last statement in the proposition.
\end{proof}

\begin{remark}
The number of coatoms in case (3) is at most $k$ (respectively, $k-1$) when $w(n)=1$ (respectively, $w(k)=1$) and the following examples attain the maximum:
\[
\begin{split}
w&=[n-k+2,n-k+3,\dots,n,2,3,4,\dots,n-k+1,1] \qquad\text{when $w(n)=1$},\\
w&=[n-k+2,n-k+3,\dots,n,1,3,4,\dots,n-k+1,2] \qquad\text{when $w(k)=1$}.
\end{split}
\]
\end{remark}

\end{document}